\documentclass[11pt]{amsart}
\usepackage{fullpage,amssymb}
\usepackage{mathrsfs}
\usepackage{color}
\usepackage{algorithmicx}
\usepackage{algpseudocode}
\usepackage{tikz-cd}

\usepackage{color}   
\usepackage{hyperref}
\hypersetup{
    linktoc=all,     
    linkcolor=black,  
}

\newcommand{\Vnote}[1]{\begin{quote}{\color{red}{{\sf Visu's Note:} {\sl{#1}}}} \end{quote}}

\newtheorem{theorem}{Theorem}[section]
\newtheorem{corollary}[theorem]{Corollary} 
\newtheorem{lemma}[theorem]{Lemma}
\newtheorem{proposition}[theorem]{Proposition}

\theoremstyle{definition}
\newtheorem{definition}[theorem]{Definition}
\newtheorem{remark}[theorem]{Remark}
\newtheorem{example}[theorem]{Example}

\newcommand{\R}{{\mathbb R}}
\newcommand{\Rep}{{\rm Rep}}
\newcommand{\SI}{{\rm SI}}

\newcommand{\ST}{{\rm ST}}
\newcommand{\PD} {{\rm PD}}
\newcommand{\mlt}{{\rm mlt}}
\newcommand{\Lie} {{\rm Lie}}

\newcommand{\C}{{\mathbb C}}
\newcommand{\N}{{\mathbb N}}

\newcommand{\Z}{{\mathbb Z}}
\newcommand{\tnr}{\operatorname{Tr}}

\newcommand{\Mat}{\operatorname{Mat}}
\newcommand{\GL}{\operatorname{GL}}
\newcommand{\SL}{\operatorname{SL}}
\newcommand{\Mod}{\operatorname{mod}}

\newcommand{\Sym}{S}

\title{Maximum Likelihood Estimation for matrix normal models via quiver representations}
\author{Harm Derksen and Visu Makam}
\thanks{The first author was partially supported by NSF grants IIS-1837985 and DMS-2001460. The second author was partially supported by NSF grants DMS-1638352 and CCF-1900460.}

\begin{document}
\maketitle

\begin{abstract}
In this paper, we study the log-likelihood function and Maximum Likelihood Estimate (MLE) for the matrix normal model for both real and complex models.  We describe the exact number of samples needed to achieve (almost surely) three conditions, namely a bounded log-likelihood function, existence of MLEs, and uniqueness of MLEs.  As a consequence, we observe that almost sure boundedness of log-likelihood function guarantees almost sure existence of  an MLE, thereby proving a conjecture of Drton, Kuriki and Hoff \cite{Drton-etal}. The main tools we use are from the theory of quiver representations, in particular, results of Kac, King and Schofield on canonical decomposition and stability. 
\end{abstract}

\tableofcontents

\section{Introduction}
The following problem is fundamental in statistics in a variety of settings: Among a collection of probability distributions (a.k.a. a statistical model), find the one that best fits some empirical data. A probability distribution in the collection that maximizes the likelihood of the empirical data is called a Maximum Likelihood Estimate (MLE). Understanding the existence and uniqueness of  MLEs is an important problem that is widely studied. A related problem is to understand when the likelihood function (or equivalently the log-likelihood function) is bounded. 

In many settings, data is observed in two domains, and hence observations are naturally matrix-valued. For such observations, one sometimes assumes that they follow a {\em matrix normal distribution}. These matrix normal models have been used for various purposes in various settings, for example to EEG/MEG data \cite{Ros4, Ros6, Ros10, Ros21}, Environmental data \cite{Ros8,Ros18}, Netflix movie rating data \cite{Ros1} and facial recognition \cite{Ros25} to name a few. The existence and uniqueness of MLEs, and the boundedness of the likelihood function in matrix normal models for small sample sizes is important to understand so as to use them effectively in applications (see \cite{Drton-etal} for more motivation). This problem was studied in many manuscripts before, e.g., \cite{Dut99, Ros, Srivastava, Drton-etal}, each making partial progress.\footnote{We caution the reader that some of these papers contain erroneous results, and we refer the reader to \cite{AKRS, Drton-etal} for the state of art results prior to our work.} 

More recently,  Amendola, Kohn, Reichenbach and Seigal \cite{AKRS} uncovered connections between MLEs and stability notions in invariant theory and adapted results of B\"urgin and Draisma \cite{BD06} to further improve the results on sample size required for the (almost sure) boundedness of the log-likelihood function. It is worthwhile to mention some interesting connections to invariant theory even though it will not be relevant to our paper. The flip-flop algorithm for computing an MLE for matrix normal models \cite{Dut99, LZ05} is very similar to the algorithm proposed by Gurvits \cite{Gurvits} for computing capacity of completely positive operators, which is well known to be equivalent to null cone membership (a central problem in algorithmic invariant theory) for the so called left-right action. Even more curiously, the notion of geodesic convexity that has played a major role in understanding invariant theoretic algorithms in recent years \cite{BFGOWW} can already be seen in Wiesel's work \cite{Wiesel} several years prior in the setting of the flip-flop algorithm.

\subsubsection*{\bf Summary of our main results} In this paper, for matrix normal models (both real and complex), we will compute the {\em exact} number of samples needed to achieve (almost surely\footnote{This just means that the set of all empirical data for which the condition is not satisfied has lebesgue measure zero.}) three conditions, i.e., (1) a bounded log-likelihood function, (2) existence of MLEs and (3) uniqueness of MLEs, thereby completely resolving the problems. In particular, we prove a conjecture of Drton, Kuriki and Hoff \cite{Drton-etal} that almost sure boundedness of the likelihood function implies almost sure existence of an MLE. We utilize heavily the connections between invariant theory and MLEs and in particular the connection between matrix normal models and geometric invariant theory for quiver representations discovered in \cite{AKRS}. Our techniques, however, are significantly different from any of the previous work on these problems and rely on the algebraic aspects of theory of quiver representations. We also study a related model called the model of proportional covariance matrices and give complete answers to the aforementioned questions in that case as well. \\

Before we get into precise definitions and results, a few remarks on background literature. We refer to \cite{AKRS, Drton-etal} and references therein for more details regarding real and complex Gaussian models, matrix normal models, their MLEs and associated thresholds as well as more motivation for the problems we discuss in this paper. A detailed explanation and proofs of the connections between MLEs and invariant theory can be found in \cite{AKRS}. We point the reader to the book \cite{DW-book} as a comprehensive introductory text on quiver representations. 

\subsection{Real Gaussian models}
We denote by $\PD_n$, the cone of $n \times n$ positive definite matrices with entries in $\R$, the field of real numbers. For an $n$-dimensional Gaussian distribution with mean $0$ and covariance matrix $\Sigma \in \PD_n$, the density function is described by 
$$
f_\Sigma(y) = \frac{1}{\sqrt{\det(2 \pi \Sigma)}} e^{- \frac{1}{2} y^\top \Sigma^{-1} y}
$$
The inverse of the covariance matrix, i.e., $\Sigma^{-1}$ is called the concentration matrix and denoted $\Psi$. A subset of $\mathcal{M} \subseteq \PD_n$ defines a statistical model consisting of the $n$-dimensional Gaussian distributions with mean $0$ and concentration matrix $\Psi \in \mathcal{M}$. For Gaussian models, the data is a tuple of vectors $Y = (Y_1,\dots,Y_m) \in (\R^n)^m$, where $m$ denotes the sample size. The likelihood function $L_Y: \PD^n \rightarrow \R$ is given by 
$$
L_Y(\Psi) = \prod_{i=1}^m f_{\Psi^{-1}}(Y_i) = \det\left(\frac{\Psi}{2 \pi}\right)^{m/2} e^{-\frac{1}{2} \sum_{i=1}^m Y_i^\top \Psi Y_i}.
$$


The log-likelihood function $l_Y: \PD_n \rightarrow \R$ (upto an additive constant) is given by
$$
l_Y(\Psi) = \frac{m}{2}\log \det(\Psi) - \frac{1}{2}\tnr \left(\Psi \sum_{i=1}^m Y_iY_i^\top \right).
$$

A Maximum Likelihood Estimate (MLE) is a point $\widehat{\Psi} \in \mathcal{M}$ that maximizes the likelihood of observing the data $Y$, which is equivalent to maximizing the log-likelihood function $l_Y$. In other words, $\widehat{\Psi}$ is an MLE if $l_Y(\widehat{\Psi}) \geq l_Y(\Psi)$ for all $\Psi \in \mathcal{M}$. If the log-likelihood function is unbounded, then of course MLEs do not exist. But even when the log-likelihood function is bounded, it is not entirely obvious that MLEs exist because the supremum of the log-likelihood function may not be achieved by any particular concentration matrix. Finally, even when an MLE exists, there is no guarantee that it is unique as there may be many points in the model that achieve the maximum possible value of the log-likelihood function. 

For a Gaussian model $\mathcal{M} \subseteq \PD_n$, we define three threshold functions as follows:
\begin{enumerate} 
\item We define $\mlt_b(\mathcal{M})$ to be the smallest integer $m$ such that for $d \geq m$, the log-likelihood function $l_Y$ for $Y = (Y_1,\dots,Y_m) \in (\R^n)^d$ is bounded almost surely.
\item We define $\mlt_e(\mathcal{M})$ to be the smallest integer $m$ such that for $d \geq m$, an MLE exists almost surely for $Y \in (\R^n)^d$.
\item We define $\mlt_u(\mathcal{M})$ to be the smallest integer $m$ such that for $d \geq m$, there almost surely exists a unique MLE for $Y \in (\R^n)^d$.
\end{enumerate}

In the above, almost surely means that the property holds away from a subset of $(\R^n)^d$ of Lebesgue measure zero. We will refer to $\mlt_b,\mlt_e$ and $\mlt_u$ as maximum likelihood threshold functions. By the above discussion, we observe that $\mlt_b \leq \mlt_e \leq \mlt_u$.

\subsection{Complex Gaussian models}
The setting of complex Gaussian models is very much analogous, with minor changes. The density function for a complex $n$-dimensional Gaussian with mean $0$ and covariance matrix $\Sigma \in \PD_n$ (the cone of positive definite $n \times n$ complex matrices) is given by
$$
f_\Sigma(y) = \frac{1}{\sqrt{\det(2 \pi \Sigma)}} e^{- \frac{1}{2} y^\dag \Sigma^{-1} y},
$$
where $y^{\dag}$ denotes the adjoint of $y$, i.e., conjugate transpose.

The log-likelihood function is given by
$$
l_Y(\Psi) = \frac{m}{2} \log \det(\Psi) - \frac{1}{2}\tnr \left(\Psi \sum_{i=1}^m Y_iY_i^\dag \right),
$$
where $Y_i^\dag$ denotes the adjoint of $Y_i$. The rest of the discussion follows analogously and we do not feel the need to repeat it.

\begin{remark}
We will reuse the same notation for real and complex models (for e.g., $\PD_n, l_Y$, etc). It will always be clear whether we are in a real or a complex model, so there will be no scope for confusion.
\end{remark}

\subsection{Matrix normal models and main results} \label{sec:intro-mnm}
If $n = pq$, we consider the subset 
$$
\mathcal{M}(p,q) = \{\Psi_1 \otimes \Psi_2\ |\ \Psi_1 \in \PD_p, \Psi_2 \in \PD_q\} \subseteq \PD_{pq},
$$ 
where $\otimes$ denotes the Kronecker (or tensor) product of matrices. Such a statistical model is called a matrix normal model. Sometimes it is also called a Kronecker covariance model. We will consider and deal with both real and complex matrix normal models. As mentioned above, $\PD_n$ denotes positive definite real or complex matrices depending on whether we are working with real or complex matrix normal models. When we want to differentiate between the real and complex models, we will use $\mathcal{M}_\R(p,q)$ and $\mathcal{M}_\C(p,q)$ respectively. For matrix normal models, it will convenient to interpret the data as a $p \times q$ matrix, rather than a vector of size $pq$, and we will do so.

Drton, Kuriki, and Hoff \cite{Drton-etal} suggest that an exact formula for maximum likelihood thresholds may be complicated because of the following behavior. For a sample size of two (i.e., $Y = (Y_1,Y_2) \in \Mat_{p,q}^2$, where $\Mat_{p,q}$ denotes the space of $p \times q$ matrices), consider the matrix normal model $\mathcal{M}(p,q)$.

\begin{itemize}
\item If $(p,q) = (5,4)$, then we almost surely have a unique MLE;
\item If $(p,q) = (6,4)$, then we almost surely have an MLE that is not unique;
\item If $(p,q) = (7,4)$, then MLEs do not exist;
\item If $(p,q) = (8,4)$, then we almost surely have an MLE that is not unique.
\end{itemize}


We also obtain {\em exact} formulas for the $\mlt_b,\mlt_e$ and $\mlt_u$ for matrix normal models $\mathcal{M}(p,q)$, both real and complex. There are some delicate differences between real matrix normal models and complex matrix normal models which we will elaborate on later, see also \cite[Example~5.6]{AKRS}. Nevertheless, the maximum likelihood threshold functions are the same for both real and complex matrix normal models.

From the point of view of quiver representations, it is natural to fix the number of samples, and then study the boundedness of log-likelihood function, existence and uniqueness of MLEs as $p$ and $q$ vary. This subtle change in point of view offers a significantly different perspective from earlier work. The added advantage is that the answer comes out very crisp!

\begin{theorem} \label{theo-main-mnm}
Suppose $K = \R$ or $\C$. Let $Y = (Y_1,\dots,Y_m) \in \Mat_{p,q}^m(K)$. Let $d = {\rm gcd}(p,q)$. Then, for the matrix normal model $\mathcal{M}_K(p,q)$:
\begin{enumerate}
\item If $p^2 + q^2 - mpq < 0$, then there almost surely exists a unique MLE;
\item If $p^2 + q^2 - mpq = 0$ or $d^2$, then an MLE exists almost surely. Further this MLE is almost surely unique  if and only if $d = 1$;
\item In all other cases, log likelihood function is unbounded always (not just almost surely). Consequently MLEs do not exist. 
\end{enumerate}
\end{theorem}

At this juncture, we invite the reader to verify that the complicated behavior in the examples mentioned above is consistent with the statement of the above theorem. We reformulate Theorem~\ref{theo-main-mnm} to compute exactly $\mlt_b,\mlt_e,$ and $\mlt_u$ for $\mathcal{M}(p,q)$. While not as elegant as the formulation in the above theorem, it remains fairly simple.  

\begin{theorem} \label{theo-thresh-mnm}
Consider the (real or complex) matrix normal model $\mathcal{M}(p,q)$. Let ${\rm gcd}(p,q) = d$, and let $r = \displaystyle \frac{p^2 + q^2 - d^2}{pq}$. Then 
\begin{enumerate}
\item If $p = q = 1$, then $\mlt_b = \mlt_e = \mlt_u = 1$
\item If $p = q >1$, then $\mlt_b = \mlt_e = 1$ and $\mlt_u = 3$.
\item If $p \neq q$ and $r$ is an integer, then $\mlt_b = \mlt_e =  r$. If $d = 1$, then $\mlt_u = r$, and if $d > 1$, then $\mlt_u  = r + 1$.
\item If $p \neq q$ and $r$ is not an integer, then $\mlt_b = \mlt_e = \mlt_u = \lceil \frac{p^2 + q^2}{pq} \rceil$.
\end{enumerate}
\end{theorem}

Now, it is a simple observation to see that the Drton-Kuriki-Hoff conjecture \cite{Drton-etal} follows immediately from the above theorems:

\begin{corollary} [Drton-Kuriki-Hoff conjecture \cite{Drton-etal}] \label{Conj-DKH}
For the (real or complex) matrix normal model $\mathcal{M}(p,q)$, almost sure boundedness of (log-)likelihood function implies almost sure existence of MLE. In particular, for all $(p,q)$,  
$$\mlt_b(\mathcal{M}(p,q)) = \mlt_e(\mathcal{M}(p,q)).$$
\end{corollary}

We also consider a variant of the matrix normal model where one of the matrices is diagonal which is called the model of proportional covariance matrices (see e.g., \cite{Eriksen}). Let 
$$
\mathcal{N}(p,q) = \{\Psi \otimes D\ |\ \Psi \in \PD_p, D \in \PD_q \text{ is a diagonal matrix}\} \subseteq \PD_{pq}.
$$
This model has also been considered before \cite{Ros} in the context of maximum likelihood threshold functions. We have the following results:

\begin{theorem} \label{theo:diag.model}
Consider the model $\mathcal{N}(p,q)$. Let $r = p/q$.
\begin{enumerate}
\item If $r$ is an integer, then $\mlt_b = \mlt_e = r$. If $q = 1$, then $\mlt_u = r$ and if $q >1$, then $\mlt_u = r  + 1$.
\item If $r$ is not an integer, then $\mlt_b = \mlt_e = \mlt_u = \lceil r \rceil.$
\end{enumerate}
\end{theorem}

Once again, we see that $\mlt_b(\mathcal{N}(p,q)) = \mlt_e(\mathcal{N}(p,q))$.

\subsection{Organization}
In Section~\ref{sec:inv.thry}. we recall invariant theory and the connections to MLE for Gaussian group models. We also study stability notions when the underlying field is $\R$ or $\C$ and discuss an important result, i.e., Proposition~\ref{gen.stable.transfer} that allow us to transfer generic stability results from $\C$ (where it is easier to prove things) to $\R$ (which is more important for statistics). We discuss quiver representations, stability for quiver representations and canonical decompositions in Sections~\ref{sec:quivers},~\ref{sec:stability} and ~\ref{sec:can.dec} respectively. In Section~\ref{sec:mnm} and Section~\ref{sec:diagonal}, we bring together all the material we develop to prove our main results on maximum likelihood thresholds.

\subsection{Acknowledgements}
We would like to thank Carlos Am\'endola, Mathias Drton, Kathl\'en Kohn, Philipp Reichenbach, Anna Seigal for interesting discussions and comments on an earlier draft of this paper.  We also thank Ronno Das and Siddharth Krishna for helpful discussions.

\section{Invariant theory} \label{sec:inv.thry}
Invariant theory is the study of symmetries captured by group actions. The roots of this subject can be traced back to the masters of computation in the 19th century. At the turn of the 20th century, the work of Hilbert and Weyl brought invariant theory to the forefront of mathematics, and served to establish the foundations for modern algebra and algebraic geometry. 

The basic setting is as follows. Let $G$ be a group. A representation of $G$ is an action of $G$ on a (finite-dimensional) vector space $V$ (over a field $K$) by linear transformations. This is captured succinctly as a group homomorphism $\rho: G \rightarrow \GL(V)$. In particular, an element $g \in G$ acts on $V$ by the linear transformation $\rho(g)$. We write $g \cdot v$ or $gv$ to mean $\rho(g)v$. Throughout this paper, we will only consider the setting where $G$ is a linear algebraic group (over the underlying field $K$), i.e., $G$ is an (affine) variety, the multiplication and inverse maps are morphism of varieties, and the action is a rational action (or rational representation), i.e., $\rho: G \rightarrow \GL(V)$ is a morphism of algebraic groups.

The $G$-orbit of $v \in V$ is the set of all vectors that you can get from $v$ by applying elements of the group, i.e., 
$$
O_v := \{gv\ |\ g \in G\} \subseteq V.
$$

We denote by $K[V]$, the ring of polynomial functions on $V$ (a.k.a. the coordinate ring of $V$). A polynomial function $f \in K[V]$ is called {\em invariant} if $f(gv) = f(v)$ for all $g \in G$ and $v \in V$. In other words, a polynomial is called invariant if it is constant along orbits. The invariant ring is 
$$
K[V]^G := \{f \in K[V]\ |\ f(gv) = f(v) \ \forall\ g \in G, v \in V\}.
$$

The invariant ring has a natural grading by degree, i.e., $K[V]^G = \oplus_{d=0}^\infty K[V]^G_d$ where $K[V]^G_d$ consists of all invariant polynomials that are homogeneous of degree $d$. For $v \in V$, we denote by $\overline{O_v}$, the closure of the orbit $O_v$. 

\begin{remark} \label{topology}
To define the closure, we need to define a topology on $V$. In this paper, we will only use the fields $K = \R$ or $\C$. Hence, we will use the standard Euclidean topology on $V$ for orbit closures, unless otherwise specified. This is not standard. In literature, the topology is usually taken as the Zariski topology. We will need to use the Zariski topology at times, but we will be careful in specifying it each time. For $K = \C$, the orbit closure w.r.t. Euclidean topology agrees with the orbit closure w.r.t. Zariski topology (in the setting of rational actions of reductive groups). We caution the reader that the interplay between the Euclidean and Zariski topology can be a bit tricky at times for $K = \R$.
\end{remark}


The stabilizer of the action at a point $v \in V$ is defined to be the subgroup $G_v := \{g \in G \ |\ gv = v\}$. We make a few definitions.

\begin{definition}
Let $K = \R$ or $\C$, and let $G$ be an algebraic group (over $K$) with a rational action on a vector space $V$ (over $K$), i.e., $\rho: G \rightarrow \GL(V)$. Let $\Delta$ denote the kernel of the homomorphism $\rho$. Give $V$ the standard Euclidean topology. Then, for $v \in V$, we say $v$ is 
\begin{itemize}
\item {\em unstable} if $0 \in \overline{O_v}$;
\item  {\em semistable} if $0 \notin \overline{O_v}$;
\item  {\em polystable} if $v \neq 0$ and $O_v$ is closed;
\item  {\em stable} if $v$ is polystable and the quotient $G_v/\Delta$ is finite.
\end{itemize}
\end{definition}

We point out again that our definitions may not be quite standard because we use the Euclidean topology. However, this is the form that is most suited for our purposes. Clearly, stable $\implies$ polystable $\implies$ semistable. A point is unstable if and only if it is not semistable. Also, note for any action of $G$ on $V$, there is a natural diagonal action on the direct sum $V^m$ by $g\cdot (v_1,\dots,v_m) = (gv_1,\dots,gv_m)$ for all $g \in G$ and $v_i \in V$. Moreover, note that for any group action $\rho: G\rightarrow \GL(V)$, the notions of semistable, polystable and stable are the same whether we consider the action of $G$ or the action of $\rho(G)$.\footnote{The action of $\rho(G)$ is the obvious one -- as $\rho(G)$ is a subgroup of $\GL(V)$, it acts on $V$ by matrix-vector multiplication.}

We make another important definition:

\begin{definition} \label{defn.gen.stable}
Let $K = \R$ or $\C$, and let $G$ be an algebraic group (over $K$) with a rational action on a vector space $V$ (over $K$). Then, we say $V$ is generically $G$-semistable (resp. polystable, stable, unstable) if there is a non-empty Zariski-open subset $U \subseteq V$ such that every $v \in U$ is $G$-semistable (resp. polystable, stable, unstable).
\end{definition}

The following notion of a null cone plays a central role in computational invariant theory.

\begin{definition} [Null cone] \label{defn.nullcone}
Let $K = \R$ or $\C$, and let $G$ be an algebraic group (over $K$) with a rational action on a vector space $V$ (over $K$). Then, the null cone is defined by
$$
\mathcal{N}_G(V) := \{v \in V\ |\ 0 \in \overline{O_v}\}.
$$
In other words, the null cone consists of all the unstable points in $V$.
\end{definition}

\subsection{MLE for Gaussian group models and invariant theory}
In this subsection, we will briefly recall Gaussian group models and their connections to invariant theory. Suppose $K = \R$ or $\C$. For any group $G$ acting on $K^n$ by linear transformations (i.e., $\rho: G \rightarrow \GL_n$), there is a corresponding {\em Gaussian group model} $\mathcal{M}_G := \{\rho(g)^\dag  \rho(g)\ |\ g \in G\} \subseteq \PD_n$.\footnote{Note that adjoint is the same as transpose for a matrix with real entries.} The following result for Gaussian group models was proved in \cite{AKRS} (we state a more general, but equivalent form of their result). 

\begin{theorem} [\cite{AKRS}] \label{theo:AKRS}
Let $K = \R$ or $\C$, and let $V$ be a finite dimensional Hilbert space, i.e., a vector space with a positive definite inner product (Hermitian when $K = \C$). Let $\rho:G \rightarrow \GL(V)$ be a rational action of $G$ on $V$.  Suppose $\rho(G) \subseteq \GL(V)$ is a Zariski closed subgroup, closed under adjoints and non-zero scalar multiples. Let $G_{\SL} \subseteq G$ be a subgroup such that $\rho(G_{\SL}) = \rho(G) \cap (\SL(V))$ and let $Y \in V^m$ be an $m$-tuple of samples. Then, for the (diagonal) action of $G_{\SL}$, we have
\begin{itemize}
\item  $Y$ is semistable $\Longleftrightarrow l_Y$ is bounded from above; 
\item $Y$ is polystable $\Longleftrightarrow$ an MLE exists;
\item $Y$ is stable $\implies$ there is a unique MLE. Further, if $K = \C$, the converse also holds, i.e., there is a unique MLE $\implies Y$ is stable.
\end{itemize}
\end{theorem}

\begin{remark}
In the above result, it suffices to ask for $\rho(G_{\SL})$ and $\rho(G) \cap (\SL(V))$ to have the same identity component since the stability notions for the action of either group will be the same. Indeed, we will need this mild generalization in Theorem~\ref{theo:AKRS-LR} and Proposition~\ref{AKRS-N} below.
\end{remark}

Matrix normal models are Gaussian group models. Consider the so-called {\em Left-Right action} of $G = \GL_p \times \GL_q$ on $V = \Mat_{p,q}$ given by the formula $(P,Q) \cdot Y = PYQ^{-1}$. The Gaussian group model $\mathcal{M}_G = \mathcal{M}(p,q)$. In this case, we can take $G_{\rm SL}$ to be the subgroup $\SL_p \times \SL_q$. This puts us squarely in the setup of semi-invariants for Kronecker quivers, which we will discuss in detail in later sections.

\begin{theorem} [\cite{AKRS}] \label{theo:AKRS-LR}
Let $K = \R$ or $\C$. Let $Y \in \Mat_{p,q}^m$ be an $m$-tuple of matrices. Consider the left-right action of $G_{\SL} = \SL_p \times \SL_q$ on $\Mat_{p,q}^m$. Then, w.r.t. the matrix normal model $\mathcal{M}(p,q)$,
\begin{itemize}
\item  $Y$ is $G_{\SL}$-semistable $\Longleftrightarrow l_Y$ is bounded from above; 
\item $Y$ is $G_{\SL}$-polystable $\Longleftrightarrow$  an MLE exists;
\item $Y$ is $G_{\SL}$-stable $\implies$ there is a unique MLE. Further, if $K = \C$, the converse also holds, i.e., there is a unique MLE $\implies Y$ is $G_{\SL}$-stable.
\end{itemize}
\end{theorem}

The example below is a concrete illustration of the ideas in this paper that we will use to prove our main result, i.e., Theorem~\ref{theo-main-mnm}.

\begin{example}
We take $m = 2$, $p = 4$, and $q = 7$. For generic $(Y_1,Y_2) \in \Mat_{4,7}^2$, we claim (and justify later) that there is a change of basis (on the left and right) such that both $Y_1$ and $Y_2$ are simultaneously in a block form as below (where all non-starred entries are $0$):

$$
\newcommand*{\tempb}{\multicolumn{1}{|c}{}}
\left[
\begin{array}{ccccccc}
\ast &\ast & \tempb \\ \cline{1-4}
 &  &  \tempb $\ast$ & \ast & \tempb \\ \cline{3-7}
 & & & &\tempb $\ast$ & \ast & \ast  \\
 & & & & \tempb  $\ast$ & \ast & \ast   \\
\end{array}\right]
$$
So, without loss of generality, let us assume $Y_1,Y_2$ are in the form above.

We write ${\rm diag}(d_1,\dots,d_k)$ to represent a diagonal $k \times k$ matrix with diagonal entries $d_1,\dots,d_k$. For $t \neq 0$, consider $\lambda(t) = {\rm diag}(t^7,t^7,t^{-7},t^{-7}) \in \SL_4$ and $\mu(t) = {\rm diag}(t^6,t^6,t^6,t^6,t^{-8},t^{-8},t^{-8}) \in \SL_7$. Let $g(t) = (\lambda(t),\mu(t)) \in \SL_4 \times \SL_7$. Then, one can check that $g(t) \cdot Y_i  = \lambda(t) Y_i \mu(t)^{-1} = t Y_i$ follows from the pattern of zeros. Hence $\lim_{t \to 0} g(t) \cdot Y_i = 0$. This would mean that $Y = (Y_1,Y_2)$ is not semistable because the origin is a limit point of its $\SL_4 \times \SL_7$ orbit. In other words, for the matrix normal model $\mathcal{M}(4,7)$, MLEs do not exist if you only have two samples, which agrees with the observations in Section~\ref{sec:intro-mnm} due to Drton, Kuriki and Hoff.

The fact that a generic $2$-tuple of $4 \times 7$ matrices can be simultaneous block form as mentioned above is a special case of the notion of canonical decomposition (for the $2$-Kronecker quiver) which we discuss in Section~\ref{sec:can.dec}. The ability to drive the matrices $Y_i$ to the origin in the limit using only diagonal group elements and that too of a very particular form is an exhibition of the Hilbert--Mumford criterion, see Theorem~\ref{theo:HM-crit}. In fact, this style of elementary argument could be used to prove Corollary~\ref{lin.ind.unstable}. 
\end{example}

The model of proportional covariance matrices $\mathcal{N}(p,q)$ is also a Gaussian group model. Consider the action of the group $H = GL_p \times {\rm T}_q$ on $V = \Mat_{p,q}$ given again by $(P,Q) \cdot Y = PYQ^{-1}$, where ${\rm T}_q \subseteq \GL_q$ denotes the subgroup of diagonal $q \times q$ matrices (i.e., a $q$-dimensional complex torus). It is easy to observe that the Gaussian group model $\mathcal{M}_H = \mathcal{N}(p,q)$. Further, in this case, one can take $H_{\rm SL}$ to be the subgroup $\SL_p \times \ST_q$ where $\ST_q$ denotes the subgroup of diagonal $q \times q$ matrices with determinant $1$. We remark here that fits the setup of semi-invariants for star quivers (details in later sections). We record this result to use in later sections.

\begin{proposition}  \label{AKRS-N}
Let $K = \R$ or $\C$. Let $Y \in \Mat_{p,q}^m$ be an $m$-tuple of matrices. Consider the aforementioned action of $H_{\SL} = \SL_p \times \ST_q$ on $\Mat_{p,q}^m$. Then, w.r.t. the model $\mathcal{N}(p,q)$,
\begin{itemize}
\item  $Y$ is $H_{\SL}$-semistable $\Longleftrightarrow l_Y$ is bounded from above; 
\item $Y$ is $H_{\SL}$-polystable $\Longleftrightarrow$ an MLE exists;
\item $Y$ is $H_{\SL}$-stable $\implies$ there is a unique MLE. Further, if $K = \C$, the converse also holds, i.e., there is a unique MLE $\implies Y$ is $H_{\SL}$-stable.
\end{itemize}
\end{proposition}

\subsection{Invariant theory over $\C$}
For this section, we take $K = \C$ and discuss a few notions in invariant theory. An algebraic group $G$ (over $\C$) is called a reductive group if every rational representation is completely reducible, i.e., it can be decomposed into a direct sum of irreducible representations. There are other equivalent definitions of reductive groups over $\C$. For rational actions of reductive groups, Hilbert \cite{Hilbert1,Hilbert2} showed that the invariant ring is finitely generated. 
 
\begin{remark} \label{reductive.egs}
 The groups $\GL_n = \GL_n(\C)$, $\SL_n  = \SL_n(\C)$, and finite groups are all reductive groups (over $\C$). Direct products of reductive groups are reductive, in particular $G = \GL_{n_1} \times \dots \times \GL_{n_d}$ is reductive. For any $\sigma = (\sigma_1,\dots,\sigma_d) \in \Z^{d}$, the subgroup $G_\sigma = \{(g_1,\dots,g_d) \in G\ |\ \prod_{i=1}^d \det(g_i)^{\sigma_i} = 1 \} \subseteq G$ is also a reductive group. All groups that we consider in this paper fall into the list of aforementioned examples. With reference to Theorem~\ref{theo:AKRS}, note that any complex Zariski closed subgroup of $\GL_n$ that is self-adjoint is a complex reductive group.
\end{remark}

The first point to note about orbit closures is that since invariant polynomials are continuous (w.r.t. either Zariski or Euclidean topology), any invariant polynomial will be constant not just along orbits, but their closures as well. Hence, any invariant polynomial will not be able to distinguish two points $v,w \in V$ if their orbit closures intersect. The converse is also true for rational actions of complex reductive groups (see e.g., \cite[Lemma~3.8]{Hoskins} for a proof).

\begin{theorem} [Mumford] \label{theo:mum}
Suppose $G$ is a (complex) reductive group with a rational action on $V$. Then for $v,w \in V$,
$$
\overline{O_v} \cap \overline{O_w} \neq \emptyset \Longleftrightarrow f(v) = f(w) \ \forall f \in \C[V]^G.
$$
In fact a more general statement is true -- if $W_1,W_2$ are $G$-invariant\footnote{This just means that $g W_i = W_i$ for all $g \in G$.} Zariski-closed subsets with an empty intersection, then there exists $f \in \C[V]^G$ such that $f(W_1) = 0$ and $f(W_2) = 1$.
\end{theorem}


While the definition of the null cone (see Definition~\ref{defn.nullcone}) is analytic in nature, it happens to be an algebraic variety when we consider rational actions of reductive groups. For a collection of polynomials $\{f_i\} \subseteq \C[V]$, we denote by $\mathbb{V}(\{f_i\})$ the common zero locus of all the $f_i$'s.

\begin{lemma} \label{null.cone.Z-closed}
Let $V$ be a rational representation of a (complex) reductive group $G$. Then
$$
\mathcal{N}_G(V) = \mathbb{V} \left(\bigcup_{d \geq 1} \C[V]^G_d\right).
$$
\end{lemma}

\begin{proof}
This follows immediately from Theorem~\ref{theo:mum}
\end{proof}

An important result in understanding the null cone is the Hilbert--Mumford criterion which says that you can detect whether a point $v \in V$ is in the null cone using $1$-parameter subgroups of $G$. A $1$-parameter subgroup of $G$ is simply a morphism of algebraic groups $\lambda: \C^* \rightarrow G$.

\begin{theorem} [Hilbert--Mumford criterion] \label{theo:HM-crit}
Let $G$ be a (complex) reductive group with a rational action on $V$. Then $v \in \mathcal{N}_G(V)$ if and only if there is a $1$-parameter subgroup $\lambda: \C^* \rightarrow G$ such that $\lim_{t \rightarrow 0} \lambda(t)\cdot v = 0$.
\end{theorem}

\begin{definition}
Let $V$ be a rational representation of a reductive group $G$. Then, we define three subsets
\begin{align*}
V^{ss} & := \{v \in V \ |\ v \text{ is $G$-semistable}\}, \\ 
V^{ps} & := \{v \in V \ |\ v \text{ is $G$-polystable}\}, \\
V^{st} & := \{v \in V \ |\ v \text{ is $G$-stable}\}.
\end{align*}
We call $V^{ss}$ (resp. $V^{ps}, V^{st}$) the semistable (resp. polystable, stable) locus. We will write $V^{G\text{-}ss}, V^{G\text{-}ps}, V^{G\text{-}st}$ if we need to clarify the group. 
\end{definition}

Since the semistable locus is precisely the complement of the null cone, the following is immediate from Lemma~\ref{null.cone.Z-closed}:

\begin{corollary} \label{ss-locus-open}
Let $V$ be a rational representation of a complex reductive group $G$. Then, the semistable locus $V^{ss}$ is Zariski-open (but may be empty).
\end{corollary}

Similar statements are true for the polystable and stable loci. 

\begin{lemma} \label{ps-st-locus-cons}
Let $V$ be a rational representation of a complex reductive group $G$. Then, the stable locus $V^{st}$ is Zariski open and the polystable locus $V^{ps}$ is Zariski-constructible, i.e., it is a union of Zariski locally closed subsets. 
\end{lemma}

\begin{proof}
For a non-negative integer $r$, define $\mathcal{Z}_r := \{v \in V \ |\ \dim O_v \leq r\}$. It is easy to show that $\mathcal{Z}_r$ is Zariski closed for all $r$. Let $k = \dim G/\Delta$ (where $\Delta$ is the kernel of the action), then clearly $\mathcal{Z}_k = V$ since no orbit can have dimension larger than $k$.

If $V^{st}$ is empty, then it is Zariski open. If $V^{st}$ is non-empty, there is some $w \in V^{st}$. Hence $G_w/\Delta$ is finite, which means that $\dim O_w = k$ and $O_w$ is Zariski-closed (see Remark~\ref{topology}). Hence $O_w$ and $\mathcal{Z}_{k-1}$ are $G$-invariant Zariski-closed subsets with empty intersection, so by Theorem~\ref{theo:mum}, we have $f \in \C[V]^G$ such that $f(O_w) = 1$ and $f(\mathcal{Z}_{k-1}) = 0$. Consider $V_f :=. \{v \in V | f(v) \neq 0\}$. Then, clearly $V_f \cap \mathcal{Z}_{k-1} = \emptyset$ because $f(\mathcal{Z}_{k-1}) = 0$. Take $v \in V_f$, we have $\dim O_v = k$ since $v \notin \mathcal{Z}_{k-1}$. Further, we claim $O_v$ is closed. If not, take $v_1 \in \overline{O}_v \setminus O_v$. Then, we must have $\dim O_{v_1} < \dim O_v = k$,\footnote{This is well known, but essentially follows from the fact that $O_v$ is Zariski constructible, so contains a dense Zariski-open subset of $\overline{O}_v$. Hence, the dimension of $\overline{O}_v \setminus O_v$ is of smaller than $\dim \overline{O}_v = \dim O_v$. Since $O_{v_1} \subseteq \overline{O}_v \setminus O_v$, we get that $\dim O_{v_1} < \dim O_v$.} so $v_1 \in \mathcal{Z}_{k-1}$. However, since $v_1 \in \overline{O}_v$, we have $f(v_1) = f(v) \neq 0$, which means that $v_1 \in V_f$, which is absurd since $V_f \cap \mathcal{Z}_{k-1}$ is empty. Hence $O_v$ is closed. In other words, $V_f \subseteq V^{st}$. To summarize, we have shown that for each $w \in V^{st}$, there is a Zariski-open subset $V_f$ such that $w \in V_f \subseteq V^{st}$, which means that $V^{st}$ is Zariski-open.

Let $\mathcal{C}_r := \{v \in \mathcal{Z}_r \ |\ v \text{ is polystable}\}$. The above argument shows that $\mathcal{C}_k = V^{st}$ is Zariski-open in $\mathcal{Z}_k = V$. A similar argument shows that $\mathcal{C}_r$ is Zariski-open in $\mathcal{Z}_r$. In particular, this means that $\mathcal{C}_r$ is a Zariski locally closed subset of $V$. Thus $V^{ps} = \cup_{i=1}^k \mathcal{C}_i$ is a union of Zariski locally closed subsets, i.e., it is Zariski-constructible.
\end{proof}

For an algebraic group $G$, we denote its connected component of the identity by $G^0$, which is an algebraic subgroup. The following result tells us that as far as stability notions are concerned, one might as well restrict themselves to the connected component of identity. 

\begin{lemma} \label{lem.get.rid.fin}
Let $V$ be a rational representation of a complex reductive group $G$. Let $G^0 \subseteq G$ denote its connected component of identity. Then for $v \in V$, $v$ is $G$-semistable/stable/polystable if and only if it is $G^0$-semistable/stable/polystable.
\end{lemma}

\begin{proof}
For semistability, observe that the Hilbert-Mumford criterion is the same whether you use $G$ or $G^0$. For polystability, use the fact that the $G$-orbit is a finite disjoint union of $G^0$-orbits, each of which forms a connected component of the $G$-orbit. For stability, using the orbit-stabilizer theorem (i.e., dimension of stabilizer + dimension of orbit = dimension of the group), we see that $\dim(G_v) = \dim((G^0)_v)$ since the $G$-orbit and $G^0$-orbit have the same dimension. Note that $G_v$ or $(G^0)_v$ are finite if and only if their dimensions are $0$. Since their dimensions are equal, one of them is finite if and only if the other is. 
\end{proof}


\subsection{Invariant theory over $\R$}
Even though invariant theory is nicest when $K = \C$, the case of $K = \R$ is perhaps more important in the context of MLE and statistics in general. Hence, in this subsection, we collect some results on invariant theory over the real numbers which will help us ``transfer'' results from $\C$ to $\R$. We also intend that this subsection serve as a general reference for the reader who is not familiar with the intricacies of invariant theory and algebraic groups over $\R$. The following definitions are from \cite{Borel-linalg}.

A complex (affine) variety $X$ with an $\R$-structure (i.e., an $\R$-subalgebra $\R[X] \subseteq \C[X]$ such that $\R[X] \otimes_\R \C = \C[X]$\footnote{For any complex (affine) variety $Y$, we denote by $\C[Y]$ its coordinate ring}) is called an (affine) $\R$-variety. As a technical point, we identify a variety $X$ with its complex points $X_\C$. A morphism $f:X \rightarrow Y$ of (affine) varieties is equivalent to a map on the coordinate rings $f^*: \C[Y] \rightarrow \C[X]$. The morphism $f$ is said to be defined over $\R$ if $f^*(\R[Y]) \subseteq \R[X]$. A complex algebraic group $G$ is called an $\R$-group if it is an (affine) $\R$-variety such that the multiplication map and inverse map are defined over $\R$, and its real points $G_\R$(i.e., ${\rm Spec}(\R[G])$) is an algebraic group over $\R$.

\begin{remark}
All varieties in this paper will be affine, so we will henceforth drop the prefix affine.
\end{remark}

For this entire section, let $G$ be a connected reductive $\R$-group. Let $V$ be a rational representation of $G$ that is defined over $\R$ -- this means that $V$ is an $\R$-variety, and that $\rho: G \rightarrow \GL(V)$ is defined over $\R$.\footnote{Equivalently, the map $G \times V \rightarrow V$ is defined over $\R$. Also note that if $V$ is an $\R$-variety, then $\GL(V)$ is also an $\R$-variety, and its real points are $\GL(V)_\R = \GL(V_\R)$.} In particular, we get an action of $G_\R$ on $V_\R$.

\begin{remark}
All the groups we consider in this paper will be connected reductive $\R$-groups, and all representations will be defined over $\R$. Groups such as $\GL_n$ and $\SL_n$ are connected reductive, and are very naturally $\R$-groups, and hence so are their direct products. The group $G_\sigma$ as defined in Remark~\ref{reductive.egs} is also naturally an $\R$-group, and is connected if $\sigma$ is indivisible, i.e., ${\rm gcd}(\sigma_1,\dots,\sigma_d) = 1$.
\end{remark}

The following proposition is a fundamental result:

\begin{proposition} \label{inv.base.change}
Let $G$ be a connected reductive $\R$-group. Let $V$ be a rational representation of $G$ that is defined over $\R$. Then, the invariant ring for the action of $G$ on $V$ is obtained by base change from the invariant ring for the action of $G_\R$ on $V_\R$, i.e.,
$$
\R[V]^{G_\R} \otimes_\R \C = \C[V]^G.
$$
In other words, ${\rm Spec}(\C[V]^G)$ (also called the categorical quotient) is naturally an $\R$-variety.
\end{proposition}

\begin{proof}
To begin with, first note that by definition $\R[V]$ is the same as $\R[V_\R]$, the coordinate ring of $V_\R$. The action $G \times V \rightarrow V$ gives a map $\mu_\C:\C[V] \rightarrow \C[G] \otimes \C[V]$. Similarly, we get a map $\mu_\R: \R[V] \rightarrow \R[G] \otimes \R[V]$, and note that $\mu_\R \otimes_\R \C = \mu_\C$ is simply a restatement that the action is defined over $\R$.

It is easy to see that the polynomial $f \in \C[V]$ is $G$-invariant if and only if $\mu_\C(f) = 1 \otimes f$. Write $f = f_1 + i f_2$ with $f_j \in \R[V]$. Then, we see that $\mu_\C(f) = \mu_\R(f_1) + i \mu_\R(f_2)$. Thus $\mu_\R(f_1)$ and $\mu_\R(f_2)$ are the real and imaginary parts of $\mu_\C(f)$ which are $1 \otimes f_1$ and $1 \otimes f_2$ respectively. So, for each $j$, we have $\mu_\R(f_j) = 1 \otimes f_j$, which means that $f_j$ is $G_\R$-invariant. This proves $\supseteq$.

To prove $\subseteq$, it suffices to prove that $\R[V]^{G_\R} = \R[V_\R]^{G_\R} \subseteq \C[V]^G$. Indeed, observe that the action of $G$ on $V$ gives an action of $G$ on $\C[V] = \Sym(V^*)$. A function $f \in \C[V]$ is $G$-invariant if and only if $g \cdot f = f$ for all $g \in G$ (easy to see and only uses that $\C$ is an infinite field). Similarly, we have an action of $G_\R$ on $\R[V_\R]$ and $f \in \R[V_\R]^{G_\R}$ if and only if $g \cdot f = f$ for all $g \in G_\R$. Now, suppose $f \in \R[V_\R]^{G_\R} \subseteq \C[V]$. Then, since $g \cdot f = f$ for all $g \in G_\R$ and $G_\R$ is Zariski-dense in $G$, we get that $g \cdot f = f$ for all $g \in G$, so $f \in \C[V]^G$. Thus $\R[V_\R]^{G_\R} \subseteq \C[V]^G$ as required.
\end{proof}

The following result is crucial for transferring our results for complex matrix normal models to real matrix normal models.

\begin{proposition} \label{prop:transfer}
Let $G$ be a connected reductive $\R$-group. Let $V$ be a rational representation of $G$ that is defined over $\R$. Let $v \in V_\R$. Then $v$ is semistable/polystable/stable for the $G_\R$-action if and only if $v$ is semistable/polystable/stable for the $G$-action.
\end{proposition}

\begin{proof}
Let us split the argument for each notion of stability.
\begin{itemize}
\item {\bf Semistability:} Suppose $v$ is $G$-semistable. Then, there is a homogeneous polynomial invariant $f \in \C[V]^G$ such that $f(v) \neq 0$. Write $f = f_1 + i f_2$ where $f_j \in \R[V]$ and hence in $\R[V]^G$ by the above proposition. Thus, $f_j(v) \neq 0$ for some $j$. By homogeneity, $f_j(0) = 0$. Since $f_i$ is $G_\R$-invariant, the $G_\R$-orbit closure of $v$ cannot contain the origin. For the converse, suppose $v \in V_\R$ is $G$-unstable. Then, by Theorem~\ref{theo:HM-crit} there is a $1$-parameter subgroup that drives $v$ to $0$ in the limit. By a result of Birkes \cite[Theorem~5.2]{Birkes}, you can choose a $1$-parameter subgroup defined over $\R$ that drives $v$ to $0$ in the limit. Hence, $v$ is $G_\R$-unstable.

\item {\bf Polystability:} Suppose the $G$-orbit of $v$ is closed in the Euclidean topology (and hence in the Zariski topology, see Remark~\ref{topology}). Borel and Harish-Chandra \cite[Proposition~2.3]{BHC} show that this implies that the $G_\R$-orbit of $v$ is closed in the Euclidean topology. Birkes showed that if the $G_\R$-orbit of $v$ is closed in the Euclidean topology, then the $G$-orbit is closed in the Zariski topology (and hence the Euclidean topology), see \cite[Corollary~5.3]{Birkes}.

\item {\bf Stability:} Stability is polystability along with the fact that the stabilizer (modulo the kernel) is finite. We already know from above that $v$ is $G$-polystable if and only if $v$ is $G_\R$-polystable. So, we only have to analyze the stabilizers. We will utilize heavily the fact that the dimensions of  Lie algebras reflect the dimensions of the groups themselves (in both the real and complex settings).\footnote{The theory of Lie algebras is well understood and we do not intend to recall the theory here. We refer the interested reader to standard references, e.g., \cite{Borel-linalg, Borel-notes, BHC, Springer}.}  Let $\Delta$ denote the kernel of the representation $\rho: G \rightarrow \GL(V)$

Since $G_v \supseteq \Delta$, we see that $G_v/\Delta$ is finite if and only if we have an equality of Lie algebras $\Lie(G_v) = \Lie(\Delta)$. Since $G_v$ (resp. $\Delta$) is defined over $\R$ (see \cite[Proposition~12.1.2, Corollary~12.1.3]{Springer}), we get that $(G_v)_\R$ (resp. $\Delta_\R$) is a real manifold whose dimension equals the complex dimension of $G_v$ (resp. $\Delta$), see \cite[Section~5.2]{Borel-notes}. In fact, more is true, the Lie algebra of $\Lie((G_v)_\R)$ (resp. $\Lie(\Delta_\R)$) is a real form of $\Lie(G_v)$ (resp. $\Lie(\Delta)$), see \cite[Section~5.3]{Borel-notes}. In other words, we have $ \Lie((G_v)_\R) \otimes_\R \C = \Lie(G_v)$ and $\Lie(\Delta_\R) \otimes_\R \C = \Lie(\Delta)$.


Since, we have inclusions $\Delta \subseteq G_v$ (and consequently $\Delta_\R \subseteq (G_v)_\R$), we get that $\Lie(G_v) = \Lie(\Delta)$ if and only if $\Lie((G_v)_\R) = \Lie(\Delta_\R)$. It is perhaps trivial, but nonetheless necessary to observe that $(G_v)_\R = (G_\R)_v$ and $\Delta_\R$ is the kernel of $\rho_\R$.  Thus $\Lie((G_v)_\R) = \Lie(\Delta_\R)$ if and only if both $(G_\R)_v$ and $\Delta_\R$ have the same dimension if and only if $(G_\R)_v/\Delta_\R$ is finite (since $(G_\R)_v \supseteq \Delta_\R$, and real algebraic groups have finitely many components).

To summarize, we have that $G_v/\Delta$ is finite if and only if $(G_\R)_v/\Delta_\R$ is finite, so $v$ is $G$-stable if and only if it is $G_\R$-stable.

\end{itemize}
\end{proof}


We write $(V^{G\text{-}ss})_\R$ to denote the real points of the set of $G$-semistable points of $V$, i.e., $(V^{G\text{-}ss})_\R = V^{G\text{-}ss} \cap V_\R$. We write $(V_\R)^{G_\R\text{-}ss}$ to denote the $G_\R$-semistable points of $V_\R$. We will use similar notation for polystable and stable loci as well. The following is immediate from Proposition~\ref{prop:transfer}.

\begin{corollary} \label{cor:transfer}
We have $(V^{G\text{-}ss})_\R = (V_\R)^{G_\R\text{-}ss}$, $(V^{G\text{-}ps})_\R = (V_\R)^{G_\R\text{-}ps}$, and $(V^{G\text{-}st})_\R = (V_\R)^{G_\R\text{-}st}$
\end{corollary}

Recall the notions of generic $G$-semistability/polystability/stability from Definition~\ref{defn.gen.stable}.

\begin{proposition} \label{gen.stable.transfer}
Let $G$ be a connected reductive $\R$-group. Let $V$ be a rational representation of $G$ that is defined over $\R$. Then $V$ is generically $G$-semistable (resp. $G$-polystable, $G$-stable) if and only if $V_\R$ is generically $G_\R$-semistable (resp. $G_\R$-polystable, $G_\R$-stable).
\end{proposition}

\begin{proof}
Let $X \subseteq V$ be a Zariski-constructible subset. Then it is easy to see that $X$ contains a dense Zariski-open subset of $V$ if and only if $X_\R = X \cap V_\R$ contains a dense Zariski-open subset of $V_\R$. Now, the proposition follows from the fact that $V^{ss}, V^{ps}$ and $V^{st}$ are Zariski-constructible (by Corollary~\ref{ss-locus-open} and Lemma~\ref{ps-st-locus-cons}) along with Corollary~\ref{cor:transfer}.
\end{proof}

\section{Quiver representations} \label{sec:quivers}
The theory of quivers and their representations forms a rich generalization of linear algebra. Numerous applications of quivers have been discovered in various algebraic subjects ranging from cluster algebras and cluster categories \cite{DWZ, Keller}, Schubert calculus \cite{DSW, Ressayre,DW-LR}, moduli spaces, Donaldson-Thomas invariants and cohomological Hall algebras and non-commutative algebraic geometry (see \cite{Reineke} and references therein) and symplectic resolutions (see \cite{Ginzburg} and references therein) to name a few. More recently, the invariant theory of quivers has played an influential role in areas of theoretical computer science, notably to Geometric Complexity Theory and non-commutative identity testing \cite{DM,DM-arbchar, GCTV,IQS,IQS2}, Brascamp--Lieb inequalities \cite{GGOW-BL} and simultaneous robust subspace recovery \cite{CK}.

Let $K$ denote the ground field. The reader should keep in mind $K = \R$ or $\C$. A quiver $Q$ is a directed acyclic graph, i.e. a set of vertices denoted $Q_0$ and a set of arrows $Q_1$. For each arrow $a \in Q_1$, we denote by $ta$ and $ha$, the tail vertex and head vertex of the arrow. We will demonstrate all the basic notions and definitions in the crucial example (below) of the $m$-Kronecker quiver $\Theta(m)$ with two vertices $x$ and $y$ with $m$ arrows $a_1,\dots,a_m$ from $y$ to $x$. 

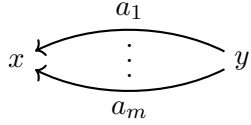
\begin{center}
\begin{tikzpicture}

\path node (x) at (0,0) []{$x$};
\path node (y) at (3,0) []{$y$};

\draw [<-,thick] (x) .. controls (1,0.5) and (2,0.5) ..
node[midway,above] {$a_1$}
(y);

\draw [<-,thick] (x) .. controls (1,-0.5) and (2,-0.5) ..
node[midway,below] {$a_m$}
(y);

\path node at (1.5,0.2) [] {.};
\path node at (1.5,0) [] {.};
\path node at (1.5,-0.2) [] {.};

\end{tikzpicture}
\end{center}

A representation $V$ of $Q$ is simply an assignment of a finite-dimensional vector space (over a ground field $K$) $V(x)$ for each $x \in Q_0$ and a linear transformation $V(a): V(ta) \rightarrow V(ha)$ for each arrow $a \in Q_1$.  A morphism of quiver representations $\phi: V \rightarrow W$ is a collection of linear maps $\phi(x): V(x) \rightarrow W(x)$ for each $x \in Q_0$ subject to the condition that for every $a \in Q_1$, the diagram below commutes.

\begin{center}
\begin{tikzcd}
V(ta) \arrow[r, "V(a)"] \arrow[d,, "\phi(ta)"]
& V(ha) \arrow[d, "\phi_{ha}"] \\
W(ta) \arrow[r, "W(a)"]
& W(ha)
\end{tikzcd}
\end{center}

A representation $V$ of $\Theta(m)$ is given by assigning vector spaces $V(x)$ and $V(y)$ to $x$ and $y$, and $m$ linear maps $V(a_1),\dots,V(a_m)$ from $V(y)$ to $V(x)$. A morphism between two representations $V$ and $W$ of $\Theta(m)$ is two linear maps $\phi(x): V(x) \rightarrow W(x)$ and $\phi(y): V(y) \rightarrow W(y)$ such that $\phi(x) \circ V(a_i) = W(a_i) \circ \phi(y)$ for all $1 \leq i \leq m$.

A subrepresentation $U$ of $V$ is a collection of subspaces $U(x) \subseteq V(x)$ such that for every edge the linear map $U(a)$ is simply a restriction of $V(a)$. In particular, this means that the image of $U(ta)$ under $V(a)$ will need to be contained in $U(ha)$. For two representations $V$ and $W$, we define their direct sum $V \oplus W$ to be the representation that assigns $V(x) \oplus W(x)$ to each vertex $x$ and the linear map $\begin{pmatrix} V(a) & 0 \\ 0 & W(a) \end{pmatrix}$ for each arrow $a \in Q_1$. Similarly, the notion of direct summand, image, kernel, co-image, etc are all defined in the obvious way, see \cite{DW-book} for details. In summary, the category of quiver representations forms an abelian category.

The dimension vector of a representation $V$ is $\underline{\dim}(V) = (\dim V(x))_{x \in Q_0}$. So, for a representation $V$ of $\Theta(m)$, its dimension vector is $\underline{\dim}(V) = (\dim(V(x)), \dim(V(y)))$.  
For any representation $V$ of a quiver $Q$, if we pick bases for $V(x)$ and $V(y)$, then this identifies $V(x)$ with $K^{\dim(V(x))}$ and $V(y)$ with $K^{\dim(V(y))}$. Further, with this identification, every linear map $V(a)$ is just a matrix of size $\dim(V(ha)) \times \dim(V(ta))$. Thus, we come to the following definition. For any dimension vector $\alpha = (\alpha(x))_{x \in Q_0} \in \N^{Q_0}$ (where $\N = \{0,1,2,\dots,\}$), we define the representation space
$$
\Rep(Q,\alpha) = \bigoplus_{a \in Q_1} \Mat_{\alpha(ha),\alpha(ta)}.
$$

Any point $V= (V(a))_{a \in Q_1} \in \Rep(Q,\alpha)$ can be interpreted as a representation of $Q$ with dimension vector as follows: for each $x \in Q_0$, assign the vector space $K^{\alpha(x)}$, and for each arrow $a \in Q_1$, the matrix $V(a)$ describes a linear transformation from $K^{\alpha(ta)}$ to $K^{\alpha(ha)}$. The base change group $\GL(\alpha) = \prod_{x \in Q_0} \GL_{\alpha(x)}$ acts on $\Rep(Q,\alpha)$ in a natural fashion where $\GL_{\alpha(x)}$ acts on the vector space $K^{\alpha(x)}$ assigned to vertex $x$. More concretely, for $g = (g_x)_{x \in Q_0} \in \GL(\alpha)$ and $V = (V(a))_{a \in Q_1} \in \Rep(Q,\alpha)$, the point $g \cdot V \in \Rep(Q,\alpha)$ is defined by the formula
$$
(g \cdot V) (a) = g_{ha} V(a) g_{ta}^{-1}.
$$

The $\GL(\alpha)$ orbits in $\Rep(Q,\alpha)$ are in $1-1$ correspondence with isomorphism classes of $\alpha$-dimensional representations.

Consider the subgroup $\SL(\alpha) = \prod_{x \in Q_0} \SL(\alpha(x)) \subseteq \GL(\alpha)$. Then, the invariant ring for the action of $\SL(\alpha)$ on $\Rep(Q,\alpha)$ is called the ring of semi-invariants
$$
\SI(Q,\alpha) = K[\Rep(Q,\alpha)]^{\SL(\alpha)}.
$$

For the $m$-Kronecker quiver $\Theta(m)$, suppose we pick a dimension vector $\alpha = (p,q)$ (we use the convention that the first entry corresponds to vertex $x$), then the representation space 
$$
\Rep(\Theta(m),(p,q)) = \Mat_{p,q}^m.
$$

Now, $\GL(\alpha) = \GL_p \times \GL_q$, and the action is given by the formula 
$$
(g_1,g_2) \cdot (Y_1,\dots,Y_m) = (g_1 Y_1g_2^{-1},\dots, g_1 Y_m g_2^{-1}).
$$

The orbits of this action correspond to isomorphism classes of $(p,q)$-dimensional representations of $\Theta(m)$. The subgroup $\SL(\alpha) = \SL_p \times \SL_q$. First, observe that $Y = (Y_1,\dots,Y_m)$ is semistable/polystable/stable (for the action of $\SL_p \times \SL_q$) if and only if $\lambda Y = (\lambda Y_1,\dots, \lambda Y_m)$ is semistable/polystable/stable for $\lambda \in K^*$. This is a simple consequence of the fact that the action is by linear transformations. Thus, we see that whether $Y = (Y_1,\dots,Y_m)$ is semistable, polystable, or stable (for the action of $\SL(\alpha)$) only depends on the isomorphism class of the quiver representation it defines (i.e., the $\GL(\alpha)$-orbit). This is the starting point of understanding the various stability notions from a representation theoretic perspective, which we will discuss in more detail in the next section. 

\begin{remark}
The space $\Rep(Q,\alpha)$ is a representation of $\GL(\alpha)$ and its various subgroups such as $\SL(\alpha)$. At the same time, we refer to a point $V \in \Rep(Q,\alpha)$ also as a representation. We advise the reader to keep in mind that we think of $V$ as a representation of the quiver $Q$ to avoid confusion. Moreover, if $V \in \Rep(Q,\alpha)_\R = \bigoplus_{a \in Q_1} \Mat_{\alpha(ha),\alpha(ta)}(\R)$, then it can be thought of as both a real and complex representation of $Q$.
\end{remark}

\subsection{Indecomposability of modules over field extensions}
Let $A$ be a finite-dimensional $\R$-algebra, and let $A\text{-}\Mod$ denote the category of finite dimensional (left)-modules over $A$. We denote by $A_\C := A \otimes_\R \C$, the $\C$-algebra obtained by extending scalars. Let $A_\C\text{-}\Mod$ denote the category of finite dimensional (left)-modules over $A_\C$. For any module $M \in A\text{-}\Mod$, let $M_\C := M \otimes_\R \C \in A_\C\text{-}\Mod$. The $A_\C$ structure on $M_\C$ is the obvious one you get by extending scalars. We can interpret $M_\C$ as an $A$-module, and as an $A$-module, we have $M_\C = M \oplus i M$ and hence $M_\C = M^{\oplus 2}$ as $A$-modules.

The Krull-Remak-Schmidt Theorem says that modules in the category $A\text{-}\Mod$ can be decomposed as a direct sum of indecomposables and this decomposition
is essentially unique in the sense that  any two such decompositions  have the same summands (counted with multiplicities). The Krull-Remak-Schmidt theorem holds for $A_\C\text{-}\Mod$ as well. Note that an indecomposable module is one that cannot be wrtiten as a direct sum of two or more (proper) submodules, and not to be confused with an irreducible (or simple) module, which is a module with no non-trivial submodules.

\begin{lemma}
Let $M \in A\text{-}\Mod$. If $M_\C \in A_\C\text{-}\Mod$ is indecomposable, then $M$ is indecomposable.
\end{lemma}

\begin{proof}
If $M = M_1 \oplus M_2$, then $M_\C = (M_1)_\C \oplus (M_2)_\C$ is a decomposition of $M_\C$ as $A_\C$-modules.
\end{proof}

We will need the following lemma.

\begin{lemma} \label{lem.3indec}
Let $M \in A\text{-}\Mod$. Suppose $M_\C$ can be written as a direct sum of three or more (non-zero) submodules (as an $A_\C$-module). Then $M$ is not indecomposable as an $A$-module.
\end{lemma}

\begin{proof}
Let $M_\C = N_1 \oplus N_2 \oplus \dots \oplus N_d$ with $d \geq 3$. Each $N_i$ is an $A_\C$-module summand, and hence an $A$-module summand as well. If we further refine the $N_i$ into a direct sum of indecomposable $A$-modules, we can write $M_\C = N'_1 \oplus N'_2 \oplus \dots N'_{d'}$ for some $d' \geq d \geq 3$, where each $N'_i$ is an indecomposable $A$-module. Suppose $M$ is indecomposable. The module $M_\C  = M \oplus i M$ as an $A$-module. Hence, by the Krull-Remak-Schmidt theorem, we know that any decomposition into indecomposables has to have exactly two summands (and each of which is isomorphic to $M$ as $A$-modules). But this contradicts the fact that $M_\C = N'_1 \oplus N'_2 \oplus \dots N'_{d'}$ is a decomposition with $d' > 2$ summands.
\end{proof}

\subsection{Quiver representations as modules over the path algebras}
Let $K$ denote the ground field. For a quiver $Q = (Q_0,Q_1)$, we will define the path algebra $KQ$. A path $p$ of length $k$ is a sequence of $k$ arrows $a_ka_{k-1}\dots a_1$ such that $t(a_{i+1}) = h(a_i)$ for $1 \leq i < k$. The head vertex of the path is $h(a_k)$ and the tail vertex is $t(a_1)$. We introduce trivial paths $e_x$ of length zero for each $x \in Q_0$ with $h(e_x) = t(e_x) = x$.

The path algebra $KQ$ is a $K$-algebra with a basis labeled by all paths in $Q$. The multiplication is as follows. For paths $p$ and $q$, their product $p \cdot q$ is the concatenation of the two paths if $tp = hq$ and $0$ otherwise. For any representation $V$ of $Q$ over $K$, we can interpret it as $KQ$-module $\oplus_{x \in Q_0} V(x)$. For $w \in V(x)$, and a path $p = a_k a_{k-1} a_{k-2} \dots a_1$, the action is given by $p \cdot w = V(a_k) V(a_{k-1})\dots V(a_1) w \in V_{hp}$ if $tp = x$, and $0$ otherwise. This is in fact an equivalence of categories, see \cite{DW-book} for details.

Of particular importance is the fact that if we take $A = \R Q$, then $A_\C = \C Q$, and hence Lemma~\ref{lem.3indec} applies.

\section{Stability notions for quiver representations} \label{sec:stability}
We follow the conventions from \cite{DW-book} for consistency. For this section, we let $K =\C$. Let $Q$ be a quiver with no oriented cycles (self loops are counted as oriented cycles). Let $\alpha$ be a dimension vector. For any $\sigma \in \Z^{Q_0}$ (which we call a weight), we have a character of $\GL(\alpha)$ which we also denote $\sigma$ by abuse of notation. The character $\sigma: \GL(\alpha) \rightarrow K^*$ is given by $\sigma((g_x)_{x \in Q_0}) = \prod_{x \in Q_0} \det(g_x)^{\sigma(x)}$. The ring of semi-invariants has a decomposition

$$
{\rm SI}(Q,\alpha) = \bigoplus_{\sigma \in \Z^{Q_0}} {\rm SI}(Q,\alpha)_\sigma,
$$

where $\SI(Q,\alpha)_\sigma = \{f \in {\rm SI}(Q,\alpha) \ |\ f(g \cdot x) = \sigma(g^{-1}) f(x)\  \forall g \in \GL(\alpha)\}$.

We define the effective cone of weights 
$$
C(Q,\alpha) := \{\sigma \in \Z^{Q_0}\ |\ {\rm SI}(Q,\alpha)_{m\sigma} \neq 0\ \text{for some } m \in \Z_{>0}\}.
$$

For a weight $\sigma$ and a dimension vector $\beta$, we define $\sigma(\beta) := \sum_{x \in Q_0} \sigma(x) \beta(x)$. We point out there that every $\sigma \in C(Q,\alpha)$ must satisfy $\sigma(\alpha) = 0$. For each $0 \neq \sigma \in C(Q,\alpha)$ that is indivisible (i.e., ${\rm gcd}(\sigma(x) : x \in Q_0) = 1$), we consider the subring
$$
\SI(Q,\alpha,\sigma) := \oplus_{m=0}^\infty \SI(Q,\alpha)_{m\sigma}.
$$

For a sincere dimension vector $\alpha$ (i.e., $\alpha(x) \neq 0 \ \forall x \in Q_0$), it turns out that this subring can also be seen as an invariant ring, i.e., $\SI(Q,\alpha,\sigma) = K[\Rep(Q,\alpha)]^{\GL(\alpha)_\sigma}$ where $\GL(\alpha)_\sigma = \{g \in \GL(\alpha)\ |\ \sigma(g) = 1\}$. Note that $\GL(\alpha)_\sigma$ is a reductive group. It is well-known that the associated projective variety ${\rm Proj}(\SI(Q,\alpha,\sigma))$ defines a moduli space for the $\alpha$-dimensional representations of $Q$, see \cite{King}.

We make a definition following King \cite{King}. We follow the convention from \cite{DW-book} which is consistent with our notational choices so far, but differs from King's original convention by a sign.

\begin{definition} [King \cite{King}] \label{crit-king}
Let $Q$ be a quiver with no oriented cycles, $V$ be a representation of $Q$ and $\sigma \in \Z^{Q_0}$ a weight such that $\sigma(\underline\dim V) = 0$.
\begin{itemize}
\item $V$ is $\sigma$-semistable if $\sigma(\beta) \leq 0$ for all $\beta \in \Z_{\geq 0}^{Q_0}$ such that $V$ contains a subrepresentation of dimension $\beta$.
\item $V$ is $\sigma$-stable if $\sigma(\beta) < 0$ for all $\beta \in \Z_{\geq 0}^{Q_0}$ (other than $0$ and $\underline{\dim}(V)$) such that $V$ contains a subrepresentation of dimension $\beta$.
\item $V$ is $\sigma$-polystable if $V = V_1 \oplus V_2 \oplus \dots \oplus V_k$ such that $V_i$ are all $\sigma$-stable representations.
\end{itemize}
\end{definition}

Observe here that any $\sigma$-stable representation must be indecomposable, i.e., it cannot be written as a direct sum of (proper) subrepresentations. Indeed, suppose $V = V_1 \oplus V_2$, then $0 = \sigma(\underline{\dim} V) =  \sigma(\underline{\dim} V_1) + \sigma(\underline{\dim} V_2)$. Hence at least one of $\sigma(\underline{\dim} V_i) \geq 0$, and hence $V$ cannot be $\sigma$-stable. Also observe that if $V$ is a direct sum $V = V_1 \oplus V_2 \oplus \dots \oplus V_k$, then $V$ is $\sigma$-semistable (or $\sigma$-polystable) if and only if all the $V_i$ are. Thus, in order to understand whether a generic representation of dimension $\alpha$ is $\sigma$-semistable/polystable/stable, it is useful to understand how it decomposes as a direct sum of indecomposables, which is the topic of discussion in the next section.

We now relate $\sigma$-stability notions to $\GL(\alpha)_\sigma$-stability notions:

\begin{theorem} [King \cite{King}] \label{theo:King}
Let $Q$ be a quiver with no oriented cycles, $\alpha \in \Z_{>0}^{Q_0}$ a sincere dimension vector and $0 \neq \sigma \in C(Q,\alpha)$ an indivisible weight. A representation $V \in \Rep(Q,\alpha)$ is $\sigma$-semistable (resp. $\sigma$-polystable, $\sigma$-stable) if and only if $V$ is $\GL(\alpha)_\sigma$-semistable (resp. $\GL(\alpha)_\sigma$-polystable, $\GL(\alpha)_\sigma$-stable).
\end{theorem} 

King's original formulation is slightly different from the one above, but can be seen to be equivalent (details in Appendix~\ref{App.stability}). Now, we proceed to discuss these stability notions for the $m$-Kronecker quiver.

\subsection{Stability notions for the $m$-Kronecker quiver}
For the $m$-Kronecker quiver $\Theta(m)$, let us take $\alpha = (p,q)$. Let $p' = p/{\rm gcd}(p,q)$ and $q' = q/{\rm gcd}(p,q)$. Then 
$$
\SI(\Theta(m),(p,q)) = \bigoplus_{k=0}^\infty \SI(\Theta(m), (p,q))_{(-kq',kp')}.
$$

Indeed, observe that for $\sigma \in \Z^2$ to be in the cone of effective weights $C(\Theta(m), (p,q))$, we need $\sigma(p,q) = 0$. This means that any $\sigma \in C(\Theta(m), (p,q))$ must be a multiple of $(-q',p')$. As it turns out, non-trivial semi-invariants do not exist when you take a weight that is a negative scalar multiple of $(-q',p')$, which one can see directly from the fundamental theorem that describes semi-invariants of quivers in a determinantal fashion \cite{DW-LR, DZ, SVd} (see also \cite[Theorem~10.7.1]{DW-book}).

Further, this means that 
$$
\C[\Rep(\Theta(m), (p,q)]^{\SL_p \times \SL_q} = \C[\Rep(\Theta(m), (p,q))]^{\SL(\alpha)} = \C[\Rep(\Theta(m),(p,q)]^{\GL(\alpha)_{(-q',p')}}.
$$

In fact, we have the following result:

\begin{lemma} \label{Lemma-LR-sigma}
Let $\rho: \GL_p \times \GL_q \rightarrow \GL(\Rep(\Theta(m), (p,q))$ be the left-right action. Let $\alpha = (p,q)$ and let $p' = p/{\rm gcd}(p,q)$ and $q' = q/{\rm gcd}(p,q)$. Then $$
\rho(\SL_p \times \SL_q) = \rho(\GL(\alpha)_{(-q',p')}).
$$
In particular, this means that $\sigma$-semistability (resp. polystability, stability) for $\sigma = (-q',p')$ is the same as semistability (resp. polystability, stability) for the $\SL_p \times \SL_q$-action.
\end{lemma}

\begin{proof}
Since $\SL_p \times \SL_q \subseteq \GL(\alpha)_{(-q',p')}$, we only need to show $\supseteq$. Suppose $(g,h) \in \GL(\alpha)_{(-q',p')}$. This means that $\det(g)^{q'} = \det(h)^{p'}$. Note that $\rho(g,h) = (g \otimes (h^{-1})^\top)^{\oplus m}$. Without loss of generality, we can assume that $\det(g) = 1$ (otherwise, choose a $\lambda$ such that $\det(\lambda g ) = 1$ and replace $(g,h)$ with $(\lambda g, \lambda h)$). Thus, we have $\det(h)^{p'} = 1$. Thus $\det(h) = e^{2 \pi i n/p'}$ for some $n$. Now, choose an integer $t$ such that $t \equiv 0$ mod $q'$ and $t \equiv -n$ mod $p'$. Such a $t$ exists by the Chinese remainder theorem since $p'$ and $q'$ are coprime. Let $\mu = e^{2\pi i t/dp'q'}$, where $d = {\rm gcd}(p,q) = p/p' = q/q'$. Then $\det(\mu g) = \mu^p = e^{2\pi i t/q'} = 1$ and $\det(\mu h) = \mu^{q} \cdot e^{2\pi i n/p'}= e^{2\pi i t/p'}\cdot e^{2\pi i n/p'} = 1$. Now, observe that $\rho(\mu g, \mu h) = \rho(g,h)$ and $\mu g \in \SL_p$ and $\mu h \in \SL_q$.

\end{proof}

From Lemma~\ref{Lemma-LR-sigma} and Theorem~\ref{theo:King}, we deduce:

\begin{corollary}
Consider the $G = \SL_p \times \SL_q$-action on $\Rep(\Theta(m),(p,q))$, and let $\sigma = (-q',p')$, where $p' = p/{\rm gcd}(p,q)$ and $q' = q/{\rm gcd}(p,q)$. A representation $V \in \Rep(\Theta(m),(p,q))$ is 
\begin{itemize}
\item semistable if and only if $\sigma(\beta) \leq 0$ for all dimension vectors $\beta$ such that $V$ has a subrepresentation of dimension $\beta$.
\item stable if and only if $\sigma(\beta) < 0$ for all dimension vectors $\beta$ (other than $0$ and $(p,q)$) such that $V$ has a subrepresentation of dimension $\beta$.
\item polystable if and only if $V$ is a direct sum of $\sigma$-stable representations.
\end{itemize}
\end{corollary}

A simple corollary of the above is the following, which will be very useful.

\begin{corollary} \label{lin.ind.unstable}
Let $V \in \Rep(\Theta(m), (p,q)).$ Suppose $V = V_1 \oplus V_2 \dots \oplus V_k$ is the decomposition of $V$ into indecomposables, and let $\beta_i = \underline{\dim}(V_i)$. If for some $i$ and $j$, $\beta_i$ and $\beta_j$ are linearly independent, then $V$ is unstable (w.r.t. $\SL_p \times \SL_q$ action). 
\end{corollary}

\begin{proof}
If $V$ is to be semistable, then $\sigma(\beta_i) \leq 0$ for all $i$ and further $\sigma(\sum_i \beta_i)  = \sigma((p,q))= 0$. This means that $\sigma(\beta_i) = 0$ for all $i$. However, the kernel of $\sigma$ is clearly $1$-dimensional, so both $\beta_i$ and $\beta_j$ cannot be in the kernel if they are linearly independent.
\end{proof}

\section{Canonical decomposition} \label{sec:can.dec}
For this section, we assume $K = \C$. Let $Q$ be a quiver with no oriented cycles and let $\alpha$ be a dimension vector. Every representation $V \in \Rep(Q,\alpha)$ can be decomposed into a direct sum $V = V_1 \oplus V_2 \oplus \dots \oplus V_k$ where each $V_i$ is an indecomposable representation. The Krull-Remak-Schmidt theorem tells us that the summands that occur in any such decomposition are isomorphic (upto permutation). Of course, this decomposition will be different for different choices of $V \in \Rep(Q,\alpha)$, but for a (non-empty) Zariski-open subset of $\Rep(Q,\alpha)$, the dimension vectors of the indecomposables in the decomposition will be the same. This brings us to the definition of canonical decomposition that was first defined by Kac.

\begin{definition} [canonical decomposition \cite{Kac,Kac2}]
We write $\alpha = \beta_1 \oplus \dots \oplus \beta_k$ and call it the canonical decomposition if a generic representation $V \in \Rep(Q,\alpha)$ decomposes as a direct sum of indecomposables whose dimension vectors are $\beta_1,\dots,\beta_k$. 
\end{definition}

The existence and uniqueness of canonical decomposition requires a little argument and we refer the reader to \cite{DW-book}. To fully understand canonical decomposition, we need to recall the notion of roots. We need to define a bilinear form on $\R^{Q_0}$. For $\alpha,\beta \in \R^{Q_0}$, we define 
$$
\left< \alpha,\beta \right> = \sum_{x \in Q_0} \alpha(x)\beta(x) - \sum_{a \in Q_1} \alpha(ta)\beta(ha).
$$

\begin{definition}
A dimension vector $\alpha$ is called a root if there is an indecomposable representation of dimension $\alpha$. A root is called real if $\left<\alpha,\alpha\right> = 1$, isotropic if $\left<\alpha,\alpha\right> = 0$ and non-isotropic imaginary if $\left<\alpha,\alpha\right> < 0$. Further, it is called a Schur root if there exists a non-empty Zariski open subset of $\Rep(Q,\alpha)$ such that every representation in it is indecomposable. Note that isotropic roots are also considered imaginary roots.
\end{definition}

For the rest of this section, we fix a quiver $Q$ with no oriented cycles. We will recall some standard results. First two lemmas that are straightforward, see \cite{DW,DW-book}.

\begin{lemma} \label{obvious}
For any Schur root $\alpha$, its canonical decomposition is $\alpha = \alpha$.
\end{lemma}

\begin{lemma}
Suppose $\alpha = \beta_1 \oplus \beta_2 \oplus \dots \oplus \beta_k$ is the canonical decomposition of $\alpha$. Then each $\beta_i$ is a Schur root. 
\end{lemma}

We now state a theorem of Schofield that will be very useful for us.

\begin{theorem} [Schofield \cite{Scho-gen}] \label{Scho-can-dec}
Suppose $\alpha = \beta_1 \oplus \beta_2 \oplus \dots \oplus \beta_k$ is the canonical decomposition for $\alpha$. Then the canonical decompostion for $m \alpha$ is 
$$
m \alpha = [m\beta_1] \oplus [m\beta_2] \oplus \dots \oplus [m\beta_k],
$$
where $[m \beta] = \beta^{\oplus m}$ if $\beta$ is a real or isotropic Schur root and $[m\beta] = m\beta$ if $\beta$ is a non-isotropic imaginary Schur root.
\end{theorem}

\begin{corollary}
Suppose $\alpha = \beta_1^{\oplus m_1} \oplus \beta_2^{\oplus m_2} \oplus \dots \oplus \beta_k^{\oplus m_k}$ is the canonical decomposition of $\alpha$. For all $i$ such that $m_i > 1$, $\beta_i$ must be a real Schur root or an isotropic Schur root.
\end{corollary}

We make a definition:

\begin{definition} \label{D.sigma.stable}
Let $\alpha$ be a dimension vector. Then, we call $\alpha$ a $\sigma$-stable (resp. $\sigma$-semistable, $\sigma$-polystable) if a generic representation of dimension $\alpha$ is $\sigma$-stable (resp. $\sigma$-semistable, $\sigma$-polystable).
\end{definition}

It is easy to see that in order for $\alpha$ to be $\sigma$-stable for any $\sigma$, it must be a Schur root. Schofield proved a result in the other direction, which will be very useful for us:

\begin{theorem} [Schofield \cite{Scho-gen}] \label{Scho-schur-stable}
Let $\alpha$ be a Schur root. Then there exists $0 \neq \sigma \in C(Q,\alpha)$ such that $\alpha$ is $\sigma$-stable. 
\end{theorem}

\begin{corollary} \label{ps-not-obvious}
Let $\alpha$ be a dimension vector and $\sigma$ be a weight. Suppose $\alpha = \beta_1 \oplus \beta_2 \oplus \dots \oplus \beta_k$ is the canonical decomposition with $\beta_i$ being $\sigma$-stable for all $i$. Then, $\alpha$ is $\sigma$-polystable. Moreover, $\alpha$ is $\sigma$-stable if and only if $k = 1$.
\end{corollary}

\begin{proof}
We have a map $\phi: \GL(\alpha) \times \prod_{i = 1}^k \Rep(Q,\beta_i) \rightarrow \Rep(Q,\alpha)$, that takes $(g, (V^{(i)})_{1 \leq i \leq k}) \mapsto g \cdot (V^{(1)} \oplus V^{(2)}\oplus \dots \oplus V^{(k)})$. The fact that $\alpha = \beta_1 \oplus \beta_2 \oplus \dots \oplus \beta_k$ is the canonical decomposition means that $\phi$ is dominant, i.e., its image ${\rm Im}(\phi)$ contains a (non-empty) Zariski open subset of $\Rep(Q,\alpha)$. 

Now, the fact that each $\beta_i$ is $\sigma$-stable means that there is a non-empty Zariski open subset $U_i \subseteq \Rep(Q,\beta_i)$ that consists of $\sigma$-stable representations.\footnote{It is also true that the subset of $\sigma$-stable representations is itself Zariski-open, see, e.g., \cite[Proposition~3.19]{Hoskins}, but here we only need that it contains a Zariski open subset.} Let $U = \GL(\alpha) \times \prod_{i=1}^k U_i$. Then for any representation $V \in \phi(U) \subseteq \Rep(Q,\alpha)$, it decomposes as a direct sum of representations of dimension $\beta_1, \dots, \beta_k$, each of which is $\sigma$-stable. Hence $\phi(U)$ consists of $\sigma$-polystable representations. Now, $\phi(U)$ is Zariski-dense in ${\rm Im}(\phi)$ which is Zariski-dense in $\Rep(Q,\alpha)$. Thus the Zariski-closure of $\phi(U)$ is $\Rep(Q,\alpha)$. Since $U$ is constructible, its image under the map $\phi$ is constructible (by Chevalley's theorem on constructible sets) and hence contains a (dense, hence non-empty) Zariski-open subset of its closure, i.e., there exists a Zariski-open subset of $\Rep(Q,\alpha)$ that is contained in $\phi(U)$. Thus $\alpha$ is $\sigma$-polystable.

That $\alpha$ is $\sigma$-stable if and only if $k = 1$ is obvious.
\end{proof}

\section{Matrix normal models} \label{sec:mnm}
Let us explicitly compute the canonical decomposition for the $m$-Kronecker quiver $\Theta(m)$.

\begin{proposition} [Canonical decomposition for the $m$-Kronecker quiver] \label{cd.kron}
Consider the $m$-Kronecker quiver $\Theta(m)$ and let $\alpha = (p,q)$ be a dimension vector, and let $d = {\rm gcd}(p,q)$.
\begin{enumerate}
\item If $p^2 + q^2 - mpq < 0$, then $\alpha$ is a (non-isotropic) imaginary Schur root and its canonical decomposition is $\alpha = \alpha$.
\item If $p^2 + q^2 - mpq = 0$, then $\frac{\alpha}{d}$ is an isotropic Schur root and the canonical decomposition is $\alpha = (\frac{\alpha}{d})^{\oplus d}$ (note that $\frac{\alpha}{d} \in \Z_{\geq 0}^{Q_0}$).
\item If $p^2 + q^2 - mpq = d^2$, then $\frac{\alpha}{d}$ is a real Schur root and the canonical decomposition is $\alpha = (\frac{\alpha}{d})^{\oplus d}$.
\item In all other cases (i.e., $p^2 + q^2 - mpq > 0$, but not equal to $d^2$), the canonical decomposition has at least two linearly independent dimension vectors.
\end{enumerate}

\end{proposition}

\begin{proof}

First, observe that for a dimension vector $\gamma = (a,b)$, we have $\left< \gamma,\gamma \right> = a^2 + b^2 - mab$. The set of roots for $\Theta(m)$ are precisely the dimension vectors $(a,b)$ such that $a^2 + b^2 - mab \leq 1$ \cite{Kac}. All real roots and non-isotropic imaginary roots are Schur \cite[Theorem~4]{Kac} (note that non-isotropic imaginary roots only occur for $m \geq 3$). To be precise, Kac shows that all real and non-isotropic imaginary root occurs in a canonical decomposition, and hence must be Schur. Isotropic roots only occur for $m = 2$, and these are precisely $(a,a)$. In this case, $(1,1)$ is Schur, but the rest are of course not Schur. Moreover, observe that any real root $(a,b)$ must be indivisible as otherwise, it would not be possible for $\left<(a,b),(a,b)\right>  = a^2 + b^2 - mab = 1$. Thus, all real and isotropic Schur roots are indivisible. Further, we can conclude that if $(a,b)$ is indivisible, then $(a,b)$ is a Schur root if and only if $a^2 + b^2 - mab \leq 1$.

Let us understand when the canonical decomposition of $(p,q)$ has at least two linearly independent dimension vectors and when it does not. If it does not have two linearly independent dimension vectors, then all the dimension vectors in the canonical decomposition must be parallel to $\alpha$, so $\alpha = \lambda_1 \alpha \oplus \lambda_2 \alpha \oplus \dots \oplus \lambda_k \alpha$ is the canonical decomposition for some scalars $\lambda_i$. This means that $\alpha$ is a scalar (not necessarily integral) multiple of a Schur root, i.e., $\lambda_1 \alpha$. So, let us now turn to understanding dimension vectors that are scalar multiples of Schur roots.


Let $p' = p/d$ and $q' = q/d$. We claim that $(p,q)$ is a  scalar (not necessarily integral) multiple of a Schur root if and only if $(p',q')$ is a Schur root. The ``if'' is obvious and we have to prove ``only if''. So, let us assume $(p,q)$ is a multiple of a Schur root. Suppose $p^2 + q^2 - mpq < 0$, then clearly both $(p,q)$ and $(p',q')$ are non-isotropic imaginary Schur roots. If $p^2 + q^2 - mpq \geq 0$, then $(p,q)$ must be a multiple of real or isotropic Schur root, and since real/isotropic Schur roots are indivisible, that Schur root must be $(p',q')$. Note that as a consequence of the above arguments, we get that $(p,q)$ is a scalar multiple of a Schur root if and only if it is an integral multiple of a Schur root.


Having proved the claim in the previous paragraph, we know that $(p,q)$ is a multiple of a Schur root if and only if $p'^2 + q'^2 - mp'q' \leq 1$ or equivalently $p^2 + q^2 - mpq = d^2$ or $\leq 0$. This is precisely the first three cases and in these cases, Theorem~\ref{Scho-can-dec} tells us precisely what the canonical decomposition has to be, depending on whether the Schur root $(p',q')$ is real, isotropic or non-isotropic imaginary. In all other cases, $(p,q)$ is not a multiple of a Schur root and as argued above its canonical decomposition will have two linearly independent dimension vectors.
\end{proof}

\subsection{Maximum Likelihood thresholds for complex matrix normal models}
Let us prove Theorem~\ref{theo-main-mnm} for the case of $K= \C$ first.

\begin{proof} [Proof of Theorem~\ref{theo-main-mnm} for $K = \C$]
For this proof, let $G_{\SL}$ denote $\SL_p \times \SL_q$. For $\sigma = (-q',p')$, we know that $\sigma$-stable/polystable/semistable is the same as $G_{\SL}$-stable/semistable/polystable by Lemma~\ref{Lemma-LR-sigma}.

For $(1)$, (resp. $(2)$), observe (by Proposition~\ref{cd.kron}) that $\alpha$ (resp. $\frac{\alpha}{d}$) are Schur roots and hence $\pi$-stable for some (indivisible) $\pi$ by Theorem~\ref{Scho-schur-stable}. Such a $\pi$ must satisfy $\pi(p,q) = 0$ and so $\pi$ must be $\sigma = (-q',p')$.\footnote{Naively, it could also have been $(q',-p')$, but this is not in $C(\Theta(m),(p,q))$ as remarked before.} Thus, from Corollary~\ref{ps-not-obvious}, we deduce that $\alpha$ is $\sigma$-stable (resp. $\sigma$-polystable) and further than in the case of $(2)$, $\alpha$ is $\sigma$-stable if and only if $d = 1$. Applying Theorem~\ref{theo:AKRS-LR}, we get the required conclusion.

 $(4)$ follows immediately by combining Corollary~\ref{lin.ind.unstable}, Proposition~\ref{cd.kron} and Theorem~\ref{theo:AKRS-LR}.
\end{proof}

\subsection{Maximum Likelihood thresholds for real matrix normal models}

\begin{lemma} \label{uni-MLE-indec}
Let $Y \in \Rep(\Theta(m),(p,q))_\R = \Mat_{p,q}^m(\R)$. If an MLE given $Y$ is unique for the real matrix normal model $\mathcal{M}(p,q)$, then $Y$ is indecomposable over $\R$. In other words, it is indecomposable when thought of as a representation over $\R$ or equivalently, an $\R Q$-module.
\end{lemma}

\begin{proof}
Let $G = \SL_p \times \SL_q$, so that $G_\R = \SL_p(\R) \times \SL_q(\R)$. If an MLE given $Y$ exists, then $Y$ is polystable. Hence, without loss of generality, assume $Y$ is polystable. Moreover, without loss of generality, assume $Y$ is a point in the $G_\R$-orbit with minimal norm. Let $(G_\R)_Y$ denote the stabilizer at $Y$. 

There is a constant $\lambda \in \R\setminus \{0\}$ such that for each $(g,h) \in (G_\R)_Y$, $\lambda(g^\top g \otimes (h^{-1}) (h^{-1})^\top) \in \PD_{pq}$ is an MLE, see \cite[Proposition~5.2, Remark~5.5]{AKRS}.\footnote{There is a very minor change in the formula because the actions we use is slightly different (yet equivalent) from the one used in \cite{AKRS}. In the left-right multiplication, we multiply on the right with inverse, whereas in \cite{AKRS}, they multiply on the right with transpose. The two actions are related by an automorphism of $\SL_q$, given by $h \mapsto (h^{-1})^{\top}$, and we modify appropriately the formula for an MLE.} Clearly $\lambda I_{pq}$ is an MLE, so if it is unique, then for all $(g,h) \in (G_\R)_Y$, we must have $g^\top g = \alpha I_p$ and $h^\top h = \alpha I_q$ for some $0 \neq \alpha \in \R$. Since $g^\top g$ and $h^\top h$ are positive definite, we must have $\alpha > 0$, and since $\det(g) = \det(h) = 1$, we must have $\alpha = 1$. In other words, we must have $g^\top g = I_p$ and $h^\top h = I_q$, i.e., $g$ and $h$ are orthogonal matrices.

Suppose on the contrary that $Y$ is decomposable over $\R$. Interpreting this as a representation over $\R$ for $\Theta(m)$, we assign $\R^p$ to the vertex $x$ and $\R^q$ to the vertex $y$, and to each arrow $a_i$, we assign the linear map $Y_i: \R^q \rightarrow \R^p$. Now, $Y$ is decomposable means that there is a decomposition $\R^p = W(x) \oplus Z(x)$ and $\R^q = W(y) \oplus Z(y)$, such that for each $i$, $Y_i(W(y)) \subseteq W(x)$ and $Y_i(Z(y)) \subseteq Z(x)$. Consider $\dim(W) = (\dim(W(x)),\dim(W(y)))$ and $\dim(Z)  = (\dim(Z(x)),\dim(Z(y)))$. Then since $Y$ is $\sigma$-semistable, we must have that $\dim(W) = (ap,aq)$ and $\dim(Z) = (bp,bq)$ by Corollary~\ref{lin.ind.unstable}.\footnote{Rigorously speaking, we should say $Y$-polystable for $G_\R$ implies $Y$ is polystable for $G_\C$. A decomposition over $\R$ can be tensored (i.e., $\otimes_\R \C)$ to get a decomposition over $\C$ with (real) dimension vectors for the decomposition over $\R$ equaling the (complex) dimension vectors for the decomposition over $\C$. And over $\C$, we know that if $Y$ is polystable, then the dimension vectors of its summands must be linearly dependent with $(p,q)$ by Corollary~\ref{lin.ind.unstable}.} Now, let $c,d \in \R_{>0}$ such that $c^ad^b = 1$ and $|c|, |d| \neq 1$ (for e.g., $c = 2$ and $d = 2^{-a/b}$). Let $g \in \SL_p(\R)$ be the linear map that is defined by $g(v) = c v$ for $v \in W(x)$ and $g(v) = dv$ for $v \in Z(x)$, and let $h \in \SL_q(\R)$ be the linear map defined by $h(v) = cv$ for $v \in W(y)$ and $h(v) = dv$ for $v \in Z(y)$. Then, it is a simple check to see that $(g,h) \in (G_\R)_Y$. However, clearly $g$ and $h$ are not orthogonal matrices because they have eigenvalues with absolute value $\neq 1$. This contradicts uniqueness of MLE by the above discussion.


Thus, $Y$ must be indecomposable over $\R$.

\end{proof}

\begin{proof} [Proof of Theorem~\ref{theo-main-mnm} for $K = \R$]
Let $G_{\SL} = \SL_p(\C) \times \SL_q(\C)$ and so $(G_{\SL})_\R = \SL_p(\R) \times \SL_q(\R)$. By using Proposition~\ref{gen.stable.transfer}, all the generic stability notions (Definition~\ref{defn.gen.stable}) carry over without any change from the case of $K = \C$ to the case of $K = \R$. 

In particular, the statements regarding boundedness of log-likelihood function and existence of MLEs also carry over from $K = \C$ to $K = \R$. The only issue arises in terms of uniqueness of an  MLE. Over the reals, stability implies uniqueness of MLEs, but not conversely. Thus, even when $\Rep(\Theta(m),(p,q))_\R$ is not generically $(G_{\SL})_\R$-stable, we might still have almost sure uniqueness of MLEs. So, we need to look at the cases where we have generic polystability but not generic stability. This happens exactly when $d \geq 2$ and $p^2 + q^2 - mpq$ is either $0$ or $d^2$. This is precisely why we proved the above lemma.

Now, suppose $p^2 + q^2 - mpq = 0$ or $d^2$ and $d \geq 3$. Then, by Lemma~\ref{lem.3indec}, we get that a generic point in $\Rep(\Theta(m),(p,q))_\R$ is decomposable over $\R$, and by Lemma~\ref{uni-MLE-indec} that MLE is not unique.

Now, suppose $p^2 + q^2 - mpq = d^2$ and $d = 2$. This is precisely the case where $(p,q) = 2 \beta$ where $\beta$ is a real Schur root. This means that there is a unique indecomposable of dimension $\beta$ and is defined over $\R$ -- this is because for a real Schur root, the representation space has a Zariski-dense orbit corresponding to this unique indecomposable \cite[Lemma~11.1.3]{DW-book}. This unique indecomposable corresponds to an $\R Q$-module (for $Q = \Theta(m))$ that we will call $W$. Take a generic point $Y \in \Rep(\Theta(m),(p,q))_\R$. Interpret this as an $\R Q$-module, which we will call $M$. Then, by genericity, we know that $M_\C \cong W_\C \oplus W_\C$ as $\C Q$-modules.\footnote{For any complex vector space $V$ defined over $\R$, a generic point in $V_\R$ can be considered a generic point in $V$ because for any Zariski open subset $U$ of $V$, its real points $U_\R$ is a Zariski open subset of $V_\R$.} Hence $M_\C \cong W^{\oplus 4}$ as $\R Q$-modules. Thus $M^{\oplus 2} \cong W^{\oplus 4}$ as $\R Q$-modules. By Krull-Remak-Schmidt theorem, we get that $M \cong W^{\oplus 2}$ as $\R Q$-modules. In other words, $Y$ is decomposable over $\R$, and hence by Lemma~\ref{uni-MLE-indec}, there is not a unique MLE.

Now, suppose $p^2 + q^2 - mpq = 0$ and $d = 2$. This happens precisely in the case of $(p,q) = (2,2)$ and $m = 2$. This is a slightly tricky case, and it turns out that we cannot claim uniqueness or non-uniqueness of MLEs generically.\footnote{The subset of points having a unique MLE is semi-algebraic and full-dimensional, but not dense.} Nevertheless, it remains that we do not have almost sure uniqueness of MLEs, see \cite[Section~4]{Drton-etal} (in particular Corollary~4.6) for a more thorough explanation of this behavior. 

Thus every statement in the case of $K = \C$ transfers to the case of $K = \R$.
\end{proof}

\section{Model of proportional covariance matrices} \label{sec:diagonal}
In this section, we will focus on the model of proportional covariance matrices $\mathcal{N}(p,q)$. Once again, we will first work with $K = \C$ (and then transfer the result for $K = \R$). For this case, we consider the quiver $\mathcal{B}(q,m)$ with vertices $x,y_1,\dots,y_q$ and $m$ arrows from each $y_i$ to $x$. The quiver $\mathcal{B}(q,1)$ is pictured below.

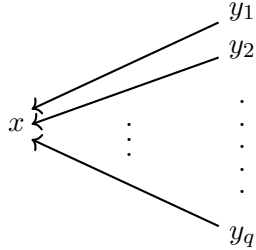
\begin{center}
\begin{tikzpicture}

\path node (x1) at (0,1.5) []{$y_1$};
\path node (x2) at (0,1) []{$y_2$};
\path node (xp) at (0,-1.5) []{$y_q$};

\path node (y) at (-3,0) []{$x$};

\draw [->,thick] (x1) -- (-2.8,0.2);
\draw [->,thick] (x2) -- (-2.8,0);
\draw [->,thick] (xp) -- (-2.8,-0.2);

\path node at (-1.5,-0.4) [] {.};
\path node at (-1.5,0) [] {.};
\path node at (-1.5,-0.2) [] {.};

\path node at (0,-0.9) [] {.};
\path node at (0,0.3) [] {.};

\path node at (0,-0.6) [] {.};
\path node at (0,-0.3) [] {.};
\path node at (0,0) [] {.};

\end{tikzpicture}
\end{center}

Let us first define an operation for quivers. For any quiver $Q = (Q_0,Q_1)$, define $Q^{[m]}$ to be the following quiver: Let its vertex set be $Q_0$, the vertex set for $Q$. For each $a \in Q_1$, define $m$ arrows $a^{[1]},\dots,a^{[m]}$ in $Q^{[m]}_1$ such that $ta = ta^{[i]}$ and $ha = ha^{[i]}$ for all $i$. Then, for any dimension vector $\alpha \in \Z_{\geq 0}^{Q_0}$, we have $\Rep(Q^{[m]},\alpha) = \Rep(Q,\alpha)^{\oplus m}$. Further, the action of $\GL(\alpha)$ on $\Rep(Q^{[m]},\alpha) = \Rep(Q,\alpha)^{\oplus m}$ is the diagonal action obtained from the action on $\Rep(Q,\alpha)$. The same holds for the action of any subgroup of $\GL(\alpha)$.

We use the convention that in a dimension vector for $\mathcal{B}(q,m)$, the coordinates correspond to $x,y_1,\dots,y_q$ in order. If we take the dimension vector $\alpha = (p,1,1,\dots,1)$, then $\Rep(\mathcal{B}(q,1),\alpha)$ can be identified with $\Mat_{p,q}$. Let $\sigma = (- q',p',\dots,p')$ where $(p',q') = \frac{1}{{\rm gcd}(p,q)} (p,q)$. Analogous to Lemma~\ref{Lemma-LR-sigma}, we can prove that $\sigma$-semistability/polystability/stability coincides with $\SL_p \times \ST_q$-semistability/polystability/stability.


Now, observe that $\mathcal{B}(q,m) = \mathcal{B}(p,1)^{[m]}$. Thus, by the above discussion, we conclude the following:

\begin{proposition}
Consider the action of $H_{\SL} = \SL_p \times \ST_q$ on $\Mat_{p,q}^m = \Rep(\mathcal{B}(q,m),p, 1,1,\dots,1))$. Let $\sigma = (-q',p',\dots,p')$ be a weight for $\mathcal{B}(q,m)$, where $(p',q') = \displaystyle\frac{1}{{\rm gcd}(p,q)} (p,q)$. 
Suppose $Y \in \Mat_{p,q}^m$. Then, $Y$ is $H_{\SL}$-semistable/polystable/stable if and only if $Y$ is $\sigma$-semistable/polystable/stable.
\end{proposition}

\begin{proof}
This is analogous to Lemma~\ref{Lemma-LR-sigma}.
\end{proof}

\begin{proposition} \label{stab.star}
Let $Q = \mathcal{B}(q,m)$, $\alpha = (p,1,1,\dots,1)$ and $\sigma = (-q',p',\dots,p')$ where $(p',q') = \frac{1}{{\rm gcd}(p,q)} (p,q)$. If $mq < p$, then every $Y \in \Rep(Q,\alpha) = \Mat_{p,q}^m$ is $\sigma$-unstable. If $mq = p$, then $\alpha$ is $\sigma$-polystable (and $\sigma$-stable precisely when $q = 1$). If $mq > p$, then $\alpha$ is $\sigma$-stable.
\end{proposition}

\begin{proof}
Let $mq < p$, and let $Y \in \Rep(Q,\alpha)$. We claim that there is a subrepresentation of dimension $\beta = (mq,1,1,\dots,1)$. This is because from each vertex $y_i$, there are $m$ arrows, and each one of them has a $1$-dimensional image. There are $mq$ of such $1$-dimensional subspaces (one for each arrow), so there is a subspace $U \subseteq \C^p$ of dimension $mq$ that contains all of these. This gives a subrepresentation of dimension $\beta$. Now $\sigma(\beta) > 0$, so $\alpha$ is not $\sigma$-semistable.

Now, let $mq = p$, and let $Y \in \Rep(Q,\alpha)$ be generic. Then, for each $y_i$, the images of the $m$ arrows starting from $y_i$ form an $m$-dimensional subspace of $\C^p$ (the vector space at the vertex $x$). You get one such $m$-dimensional subspace for each $y_i$ (call it $U_i$), hence there are $q$ of them. By genericity, we will have that $\C^q = \oplus_{i=1}^q U_i$. This means that placing $\C$ at the vertex $y_i$, $U_i$ at vertex $x$ and $\C^0$ at all other vertices gives a subrepresentation, and in fact a direct summand. Thus, $Y$ is a direct sum of $q$ indecomposables of dimensions $(m,1,0,\dots,0), (m,0,1,\dots,0),\dots, (m,0,\dots,1)$. It is straightforward to see (by genericity) that each one of these summands will be indecomposable, have no non-trivial subrepresentations, and are $\sigma$-stable. Thus, $Y$ is $\sigma$-polystable. In fact, it is easy to see that the canonical decomposition is $\alpha = (m,1,0,\dots,0) \oplus (m,0,1,\dots,0) \oplus \dots \oplus (m,0,\dots,1)$ and that each of the dimension vectors appearing in the canonical decomposition are real Schur roots that are $\sigma$-stable.

Now, let $mq > p$, and let $Y \in \Rep(Q,\alpha)$ be generic. Similar arguments as above will show that any subrepresentation has a dimension vector of the form $\beta = (\min \{m(\sum \epsilon_i), p\},\epsilon_1,\epsilon_2,\dots,\epsilon_q)$, where $\epsilon_i \in \{0,1\}$. For each subrepresentation, we observe that $\sigma(\beta) < 0$ unless $\beta = \alpha$, when $\sigma(\alpha) = 0$. Hence $Y$ is $\sigma$-stable.
\end{proof}

\begin{proof} [Proof of Theorem~\ref{theo:diag.model}]
We claim the following three statements. If $mq < p$, then log-likelihood function is unbounded. If $mq = p$, then (almost surely) an MLE exists and we have  almost sure uniqueness precisely when $m = 1$. If $mq > p$, then (almost surely) we have a unique MLE. For $K = \C$, they follow from the Proposition~\ref{stab.star} and Proposition~\ref{AKRS-N}. Transfering the result to $K = \R$ is analogous to Theorem~\ref{theo-main-mnm}. For the case $mq = p$, one has to look into the proof of Proposition~\ref{stab.star} to see that the canonical decomposition of $(p,q)$ consists of real Schur roots, so the argument parallels part $(2)$ of Theorem~\ref{theo-main-mnm}.

Reformulating this in terms of maximum likelihood thresholds gives us the required conclusion.
\end{proof}

\appendix
\section{Equivalence of stability notions} \label{App.stability}

In this appendix, we reconcile Theorem~\ref{theo:King} with King's original formulation \cite{King}. 

Let $Q$ be a quiver with no oriented cycles, let $\alpha$ be a sincere dimension vector, i.e., $\alpha(x) \neq 0$ for all $x \in Q_0$, and let $\sigma \in C(Q,\alpha)$ be a non-zero indivisible weight. Then, it is easy to see that $\C[\Rep(Q,\alpha)]^{\GL(\alpha)_\sigma} = \bigoplus_{n \in \Z} \SI(Q,\alpha)_{n\sigma}$. But, in fact, $\C[\Rep(Q,\alpha)]^{\GL(\alpha)_\sigma} =  \bigoplus_{n \in \Z_{\geq 0}} \SI(Q,\alpha)_{n\sigma}$ because we assume the quiver has no oriented cycles.\footnote{This follows from the fact that the form $\left<-,-\right>$ is non-degenerate (see \cite[Definition~2.5.3]{DW-book}) and that $\rm SI(Q,\alpha)_\gamma \neq 0$ implies that $\gamma = \left<\beta,-\right>$ for some dimension vector $\beta$ (see \cite[Theorem~10.7.1]{DW-book}). Now, since we chose $0 \neq \sigma \in \C(Q,\alpha)$, we know that for some $n \in \Z_{> 0}$, $n\sigma = \left<\beta,-\right>$ for some dimension vector $0 \neq \beta \in \Z_{\geq 0}^{Q_0}$. So for $m \in \Z_{>0}$, we get that $-m\sigma = \left<- \frac{m}{n} \beta, - \right>$, but $- \frac{m}{n}\beta$ cannot be a dimension vector as it contains negative entries, so $\SI(Q,\alpha)_{-m\sigma} = 0$ for all $m \in \Z_{>0}$.} Let $\C_\sigma$ denote the $1$-dimensional representation of $\GL(\alpha)$ corresponding to $\sigma$, i.e., $\C_\sigma = \C$ as a vector space and the linear action of $\GL(\alpha)$ is given by $g \cdot 1 = \sigma(g)$.

\begin{proposition}
Let $Q,\alpha,\sigma$ be as above. Then $V \in \Rep(Q,\alpha)$ is $\GL(\alpha)_\sigma$-semistable/polystable/stable if and only if $(V,1) \in \Rep(Q,\alpha) \oplus \C_\sigma$ is $\GL(\alpha)$-semistable/polystable/stable.
\end{proposition}

\begin{proof}
Let $z$ denote the coordinate of $\C_\sigma$ in $\Rep(Q,\alpha) \oplus \C_\sigma$. Let us split the argument for each notion of stability
\begin{itemize}
\item {\bf Semistability:} Suppose $V$ is $\GL(\alpha)_\sigma$-semistable. Then there exists $f \in \SI(Q,\alpha)_{n\sigma}$ such that $f(V) \neq 0$. This means that $\widetilde{f} = f z^n$ is $\GL(\alpha)$ invariant (with no constant term) and $\widetilde{f}(V,1) = f(V) \neq 0$. So $(V,1)$ is $\GL(\alpha)$-semistable.

Conversely, if $\widetilde{f}(V,1) \neq 0$ for some $\widetilde{f}$ that is $\GL(\alpha)$-invariant and homogeneous (say of degree $m > 0$), then write $\widetilde{f} = \sum_{i = 0}^m f_{m-i} z^{i}$, with $f_j$ homogeneous of degree $j$ for all $j$. Then, each $f_{m-i} z^{i}$ is $\GL(\alpha)$-invariant. For some $i$, we have that $f_{m-i} z^{i}$ does not vanish on $(V,1)$. So, $f_{m-i} \in \SI(Q,\alpha)_{i\sigma}$ is homogeneous of degree $(m-i)$ such that $f_{m-i}(V) \neq 0$. 
If $i = m$, this means that $f_0$ is a constant, but $f_0 \in \SI(Q,\alpha)_{m\sigma}$, which is absurd because $m\sigma \neq 0$. Thus $i < m$ and so $f_{m-i} \in \SI(Q,\alpha)_{i \sigma} \in \C[\Rep(Q,\alpha)]^{\GL(\alpha)_\sigma}$ is a homogeneous polynomial of positive degree that does not vanish on $V$. Thus, $V$ is $\GL(\alpha)_\sigma$-semistable. 

\item {\bf Polystability:} Suppose $0 \neq V$ is not $\GL(\alpha)_\sigma$-polystable. Then, by the (generalized) Hilbert--Mumford criterion (\cite[Proposition~9.6.2]{DW-book}) there exists a $1$-parameter subgroup $\lambda:\C^* \rightarrow \GL(\alpha)_\sigma$ such that $\lim_{t \to 0} \lambda(t) V = W$ where $W \notin \GL(\alpha)_\sigma \cdot V$. This means that $\lim_{t \to 0} \lambda(t) (V,1) = (W,1)$. Now, we will show that $(W,1) \notin \GL(\alpha) \cdot (V,1)$. Otherwise, we have $(W,1) = g (V,1) = (gV, \sigma(g))$ for some $g \in \GL(\alpha)$. Thus $\sigma(g) = 1$, i.e., $g \in \GL(\alpha)_\sigma$ and $gV = W$ and hence $W \in \GL(\alpha)_\sigma \cdot V$, which is a contradiction. So, $(V,1)$ is not polystable. In the case that $V = 0$, note that $(0,1)$ is not even $\GL(\alpha)$-semistable if there exists $g \in \GL(\alpha)
$ such that $|\sigma(g)| < 1$ (since that would mean $\lim_{k \to \infty} g^k (0,1) = (0,0)$). It is easy to construct such a $g \in \GL(\alpha)$ with our assumptions, i.e., $Q$ has no oriented cycles, $\alpha$ is sincere and $\sigma$ is non-zero.

Conversely, suppose $(V,1)$ is not polystable. Then there is a $1$-parameter subgroup $\lambda$ of $\GL(\alpha)$ such that $\lim_{t \to 0} \lambda(t) (V,1) = (W,c)$, with $(W,c) \notin \GL(\alpha) (V,1)$. Suppose $c = 0$, then $(W,0)$ is easily seen to be unstable because all points are unstable for the action of $\GL(\alpha)$ on $\Rep(Q,\alpha)$ if $Q$ has no oriented cycles (as is the case for us). This would mean that $(V,1)$ is not even $\GL(\alpha)$-semistable, which means that $V$ is not $\GL(\alpha)_\sigma$-semistable and hence not $\GL(\alpha)_\sigma$-polystable. Hence w.l.o.g., assume $c \neq 0$ from now on.
Now, the function $t \mapsto \sigma(\lambda(t))$ is a character of $\C^*$ and has to be of the form $t \mapsto t^k$ for some integer $k$. Since $c = \lim_{t \to 0} t^k$ is defined, we must have $k \geq 0$. If $k = 0$, we get $c = 1$ and if $k > 0$, we get $c = 0$. Since $c \neq 0$, we must have $c = 1$ and $\lambda(t) \in \GL(\alpha)_\sigma$. This means that $\lim_{t \to 0} \lambda(t) V = W$, so $W \in \overline{\GL(\alpha)_\sigma \cdot V}$. But $W \notin \GL(\alpha)_\sigma \cdot V$, because if it were, then $g V = W$ for some $g \in \GL(\alpha)_\sigma$, which means $g(V,1) = (W,1) = (W,c)$, which is a contradiction. Thus the orbit of $V$ is not closed, and hence $V$ is not polystable.
 
\item {\bf Stability:} Since $V$ is $\GL(\alpha)_\sigma$-polystable if and only if $(V,1)$ is $\GL(\alpha)$-polystable, we only need to now understand the stabilizers. First, observe that if $\Delta \subseteq \GL(\alpha)_\sigma$ is the kernel for its action on $\Rep(Q,\alpha)$, then $\Delta$ is the kernel for the action of $\GL(\alpha)$ on $\Rep(Q,\alpha) \oplus \C_\sigma$. Thus, all we need to do is to show that the two stabilizers, i.e., $\GL(\alpha)_{(V,1)}$ and $(\GL(\alpha)_\sigma)_V$, have the same dimension. In fact they are both the same. Indeed $g \in \GL(\alpha)_{(V,1)}$ if and only if $g(V,1) = (gV, \sigma(g)) =  (V,1)$ if and only if $\sigma(g) = 1$ and $gV = V$ if and only if $g \in (\GL(\alpha)_\sigma)_V$.

\end{itemize}
\end{proof}

King \cite{King} showed that $\sigma$-semistability/polystability/stability for $V \in \Rep(Q,\alpha)$ was the same as the $\GL(\alpha)$-semistability/polystability/stability of $(V,1) \in \Rep(Q,\alpha) \oplus \C_\sigma$, which we have shown is equivalent to $\GL(\alpha)_\sigma$-semistability/polystability/stability of $V$ (under the hypothesis mentioned above). Thus, the above proposition bridges the gap between the results stated in \cite{King} and Theorem~\ref{theo:King}. Finally, we remark that the hypothesis on $Q,\alpha$ and $\sigma$ cannot be entirely removed. For example, if you take $\sigma = 0$, then for $V = 0$, it is easy to see that $(V,1)$ is $\GL(\alpha)$-semistable, but $V$ is not $\GL(\alpha)_\sigma$-semistable.



\begin{thebibliography}{99}

\bibitem{Ros1} G.~I.~Allen, R.~Tibshirani, {\it Transposable regularized covariance models with an application to missing data imputation}, Ann. Appl. Stat.~{\bf 4}, no.~2, (2010), 764--790.


\bibitem{AKRS} C.~Amendola, K.~Kohn, P.~Reichenbach and A.~Seigal, {\it Invariant theory and scaling algorithms for maximum likelihood estimation}, {\tt arXiv:2003.13662}, [math.ST], 2020.
\bibitem{Borel-linalg} A.~Borel, {\it Linear algebraic groups}, Second edition. Graduate Texts in Mathematics, 126. Springer-Verlag, New York, 1991. xii+288 pp.

\bibitem{Borel-notes} A.~Borel, {\it Lie groups and linear algebraic groups. I. Complex and real groups}, Lie groups and automorphic forms, 1--49,
AMS/IP Stud. Adv. Math., 37, Amer. Math. Soc., Providence, RI, 2006.

\bibitem{BHC} A.~Borel and Harish-Chandra, {\it Arithmetic subgroups of algebraic groups}, Annals of Mathematics, Second Series~{\bf 75} (1962), 485--535.
\bibitem{Ros4} F.~Bijma, J.~C.~De Munck, R.~M.~Heethaar, {\it The spatiotemporal MEG covariance matrix modeled as a sum of Kronecker products}, NeuroImage~{\bf 27}, no.~2, (2005), 402--415.
\bibitem{Birkes} D.~Birkes, {\it Orbits of linear algebraic groups}, Annals of Mathematics, Second Series~{\bf 93} (1971), 459--475.
\bibitem{BD06} M.~B\"urgin and J.~Draisma, {\it The Hilbert null-cone on tuples of matrices and bilinear forms}, Math. Z.~{\bf 254} (2006), no.~4, 785--809.

\bibitem{BFGOWW} P.~B\"urgisser, C.~Franks, A.~Garg, R.~Oliveira, M.~Walter, and A.~Wigderson, {\it Towards a theory of non-commutative optimization: geodesic 1st and 2nd order methods for moment maps and polytopes}, FOCS 2019, 845--861.

\bibitem{CK} C.~Chindris and D.~Kline, {\it Simultaneous robust subspace recovery and semi-stability of quiver representations}, {\tt arXiv:2003.02962} [math.RT], 2020.

\bibitem{Ros6} J.~C.~De Munck, H.~M.~Huizenga, L.~J.~ Waldorp, R.~M.~Heethaar, {\it Estimating stationary dipoles from MEG/EEG data contaminated with spatially and
temporally correlated background noise}, IEEE Transactions Sign. Proc.~{\bf 50}, no.~7 (2002), 1565--1572.

\bibitem{DM} H.~Derksen and V.~Makam, {\it Polynomial degree bounds for matrix semi-invariants},  Adv. Math.~{\bf 310} (2017), 44--63.

\bibitem{DM-arbchar} H.~Derksen and V.~Makam, {\it Generating invariant rings of quivers in arbitrary characteristic}, J. Algebra~{\bf 489} (2017), 435--445.

\bibitem{DM-siq} H.~Derksen and V.~Makam, {\it Degree bounds for semi-invariant rings of quivers}, J. Pure Appl. Algebra~{\bf 222} (2018), no.~10, 3282--3292.

\bibitem{DSW} H.~Derksen, A.~Schofield and J.~Weyman, {\it On the number of subrepresentations of a general quiver representation}, J. Lond. Math. Soc. (2)~{\bf 76} (2007), no. 1, 135--147.

\bibitem{DW-LR} H.~Derksen and J.~Weyman, {\it Semi-invariants of quivers and saturation of Littlewood-Richardson co-efficients}, Journal of the American Math. Soc.~{\bf 13} (2000), 467-479.


\bibitem{DW} H.~Derksen and J.~Weyman, {\it On the canonical decomposition of quiver representations}, Compositio Math.~{\bf 133} (2002), no.~3, 245--265.
\bibitem{DW-book} H.~Derksen and J.~Weyman, {\it An introduction to quiver representations}, Graduate Studies in Mathematics~{\bf 184}, American Mathematical Society, Providence, RI, 2017. x+334 pp. 



\bibitem{DWZ} H.~Derksen, J.~Weyman and A.~Zelevinsky, {\it Quivers with potentials and their representations II: applications to cluster algebras}, J. Amer. Math. Soc.~{\bf 23} (2010), no. 3, 749--790.

\bibitem{DZ} M.~Domokos and A.~N.~Zubkov, {\it Semi-invariants of quivers as determinants}, Transformation groups~{\bf 6} (2001), 9-24.


\bibitem{Drton-etal} M.~Drton, S.~Kuriki and P.~Hoff, {\it Existence and Uniqueness of the Kronecker Covariance MLE}, {\tt arXiv:2003.06024}, [math.ST], 2020.


\bibitem{Dut99} P.~Dutilleul, {\it The MLE algorithm for the matrix normal distribution}, J. Statist. Comput. Simul.~{\bf 64} (1999), 105--123.
\bibitem{Ros8} P.~Dutilleul, B.~Pinel-Alloul, {\it A doubly multivariate model for statistical analysis of spatio-temporal environmental data}, Environmetrics~{\bf 7} (1996), 551--566.

\bibitem{Eriksen} S.~P.~Eriksen, {\it Proportionality of covariance matrices}, Ann. Statist.~{\bf 15} (1987), no.~2, 732--748.

\bibitem{GGOW}A.~Garg, L.~Gurvits, R.~Oliveira and A.~Widgerson, {\it A deterministic polynomial time algorithm for non-commutative rational identity testing}, 57th Annual IEEE Symposium on Foundations of Computer Science--FOCS 2016, 109--117, IEEE Computer Soc., Los Alamitos, CA, 2016.

\bibitem{GGOW-BL} A.~Garg, L.~Gurvits, R.~Oliveira and A.~Widgerson, {\it Algorithmic and optimization aspects of Brascamp-Lieb inequalities, via operator scaling}, Geom. Funct. Anal.~{\bf 28} (2018), no.~1, 100--145.

\bibitem{Ginzburg} V.~Ginzburg, {\it Lectures on Nakajima's quiver varieties}, Geometric methods in representation theory. I, 145--219, S\'emin. Congr., 24-I, Soc. Math. France, Paris, 2012.

\bibitem{Gurvits} L.~Gurvits, {\it Classical complexity and quantum entanglement}, Journal of Computer and System Sciences~{\bf 69} (2004), no.~3, 448--484.

\bibitem{Hilbert1}D.~Hilbert, {\it \"Uber die Theorie der algebraischen Formen}, Math. Ann.~{\bf 36} (1890), 473--534.
\bibitem{Hilbert2}D.~Hilbert, {\it \"Uber die vollen Invariantensysteme}, Math. Ann.~{\bf 42} (1893), 313--370.

\bibitem{Hoskins} V.~Hoskins, {\it Geometric invariant theory and symplectic quotients},\\ \href{http://userpage.fu-berlin.de/hoskins/GITnotes.pdf}{{\textcolor{blue}{http://userpage.fu-berlin.de/hoskins/GITnotes.pdf}}} (2012).

\bibitem{Ros10} H.~M.~Huizenga, J.~C.~De Munck, L.~J.~Waldorp, R.~P.~P.~P.~Grasman, {\it Spatiotemporal EEG/MEG source analysis based on a parametric noise covariance model}, IEEE Trans. Biomed. Eng.~{\bf 49} (2002), no.~6,  533--539.

\bibitem{IQS} G.~Ivanyos, Y.~Qiao and K.~V.~Subrahmanyam, {\it Non-commutative Edmonds' problem and matrix semi-invariants}, Comput. Complexity~{\bf 26} (2017), no.~3, 717--763.

\bibitem{IQS2} G.~Ivanyos, Y.~Qiao and K.~V.~Subrahmanyam, {\it Constructive non-commutative rank computation is in deterministic polynomial time}, Comput. Complexity~{\bf 27} (2018), no.~4, 561--593.

\bibitem{Kac} V.~Kac, {\it Infinite root systems, representations of graphs and invariant theory}, Invent. Math.~{\bf 56} (1980), no.~1, 57--92.
\bibitem{Kac2} V.~Kac, {\it Infinite root systems, representations of graphs and invariant theory. II}, J.~Algebra~{\bf 78} (1982), no.~1, 141--162.

\bibitem{Keller} B.~Keller, {\it Cluster algebras, quiver representations and triangulated categories}, Triangulated categories, 76--160, London Math. Soc. Lecture Note Ser.~{bf 375}, Cambridge Univ. Press, Cambridge, 2010.


\bibitem{King}A.~D.~King, {\it Moduli of representations of finite-dimensional algebras}, Quart. J. Math. Oxford Ser.~{\bf 45} (1994), no.~180, 515--530.


\bibitem{Matsushima} Y.~Matsushima, {\it Espaces homog\`enes de {S}tein des groupes de {L}ie complexes}, Nagoya Math. Journal~{\bf 16} (1960), 205--218.


\bibitem{GCTV}K.~Mulmuley, {\it Geometric Complexity Theory V: Equivalence between blackbox derandomization of polynomial identity testing and 
derandomization of Noether's normalization lemma}, J. Amer. Math. Soc.~{\bf 30} (2017), no.~1, 225--309.



\bibitem{LP} L.~Le~Bruyn and C.~Procesi, {\it Semisimple representations of quivers}, Trans. Amer. Math. Soc~{\bf 317} (1990), 585-598.

\bibitem{LZ05} N.~Lu and D.~L.~Zimmerman, {\it The likelihood ratio test for a separable covariance matrix}, Stat. and Prob. Letters~{\bf 73}, no.~4 (2005), 449--457.

\bibitem{Ros18} K.~V.~Mardia and C.~R.~Goodall, {\it Spatial-temporal analysis of multivariate environmental monitoring data}, Multivariate Environ. Stat.~{\bf 6} (1993), 347--386.


\bibitem{Reineke} M.~Reineke, {\it Moduli of representations of quivers}, Trends in representation theory of algebras and related topics, 589--637,
EMS Ser. Congr. Rep., Eur. Math. Soc., Zürich, 2008.

\bibitem{Ressayre} N.~Ressayre, {\it Multiplicative formulas in Schubert calculus and quiver representation}, Indag. Math. (N.S.)~{\bf 22} (2011), no. 1-2, 87--102.


\bibitem{Ros} B.~Ro\'s, B.~Fetsje, J.~C.~de Munck and Mathisca~C.~M.~de Gunst, {\it Existence and uniqueness of the maximum likelihood estimator for models with a {K}ronecker product covariance structure}, J. Multivariate Anal.~{\bf 143} (2016), 345--361.

\bibitem{Scho-gen} A.~Schofield, {\it General representations of quivers}, Proc. London Math. Soc.~(3)~{\bf 65}, (1992), no.~1, 46--64.

\bibitem{Srivastava} M.~S.~Srivastava, T.~von Rosen, D.~von Rosen, {\it Models with a Kronecker product covariance structure: estimation and testing}, Math. Methods Statist.~{\bf 17} (2008), no.~4, 357--370.

\bibitem{SVd}  A.~Schofield and M.~ van der Bergh, {\it Semi-invariants of quivers for arbitrary dimension vectors}, Indag. Mathem., N.S~{\bf 12} (2001), 125--138.

\bibitem{Springer} T.~A.~Springer, {\it Linear algebraic groups}, Reprint of the 1998 second edition. Modern Birkhäuser Classics. Birkhäuser Boston, Inc., Boston, MA, 2009. xvi+334 pp.


\bibitem{Ros21} B.~Torrésani and E.~Villaron, {\it Harmonic hidden Markov models for the study of EEG signals}, 18th European Signal Processing Conference, EUSIPCO-2010.



\bibitem{Ros25} Y.~Zhang and J.~Schneider, {\it Learning multiple tasks with a sparse matrix-normal penalty}, Adv. Neural Inf. Process. Syst.~{\bf 23} (2010), 2550--2558.

\bibitem{Wiesel} A.~Wiesel, {\it Geodesic convexity and covariance estimation}, IEEE Trans. Signal Process.~{\bf 60} (2012), no.~12, 6182--6189.

\end{thebibliography}
\end{document}